\documentclass[10pt]{article}

\title{On symplectic vortex equations over a compact orbifold Riemann surface}
\author{Hironori Sakai}
\date{}

\usepackage[slantedGreek,noBBpl]{mathpazo}
\usepackage{amsmath,amssymb,amsthm}
\usepackage[all]{xy} 
\usepackage{enumerate}
\usepackage[pdftex]{hyperref}

\newtheorem{thm}{Theorem}[section]
\newtheorem{prop}[thm]{Proposition}
\newtheorem{lem}[thm]{Lemma}
\newtheorem{cor}[thm]{Corollary}

\theoremstyle{definition}
\newtheorem{dfn}[thm]{Definition}
\newtheorem{assumption}[thm]{Assumption}

\theoremstyle{remark}
\newtheorem{remark}[thm]{Remark}
\newtheorem{example}[thm]{Example}

\newcommand{\squashup}[0]{\vspace{-6pt}}

\renewcommand{\epsilon}[0]{\varepsilon}
\renewcommand{\phi}[0]{\varphi}

\renewcommand{\tilde}[1]{\widetilde{#1}}

\renewcommand{\emph}[1]{\textbf{#1}}

\newcommand{\A}{\mathcal{A}}
\newcommand{\B}{\mathfrak{B}}
\newcommand{\C}{\mathbb{C}}
\newcommand{\F}{\mathcal{F}}
\newcommand{\G}{\mathcal{G}}
\renewcommand{\H}{\mathrm{H}}
\newcommand{\M}{\mathcal{M}}
\renewcommand{\O}{\mathcal{O}}
\renewcommand{\P}{\mathfrak{P}}
\newcommand{\R}{\mathbb{R}}
\newcommand{\U}{\mathcal{U}}
\newcommand{\V}{\mathcal{V}}
\newcommand{\X}{\mathfrak{X}}
\newcommand{\Y}{\mathfrak{Y}}
\newcommand{\Z}{\mathbb{Z}}
\def\i{\mathtt   i}  
\def\g{\mathfrak g}  

\newcommand{\MW}[1]{#1/\!\!/}
\newcommand{\Ts}{\mathcal{T}}
\newcommand{\HDR}{\mathrm{H}_\mathrm{DR}}

\newcommand{\Diff}{\mathfrak{Diff}}

\newcommand{\Bi}{\mathfrak{Bi}}

\newcommand{\StDiff}{\mathfrak{StDiff}}

\newcommand{\PG}{\mathfrak{P}^G}
\newcommand{\RG}{\mathfrak{R}^G}

\newcommand{\Cinf}{\mathcal{C}^\infty}
\newcommand{\dA}{d_{\! A}}
\newcommand{\brdA}{\overline{\partial}_{\! A}}
\newcommand{\zl}{\mu^{-1}(0)}

\DeclareMathOperator{\Hom}{Hom}
\DeclareMathOperator{\id}{id}

\DeclareMathOperator{\Mor}{Mor}
\DeclareMathOperator{\Aut}{Aut}
\DeclareMathOperator{\Vect}{Vect}
\DeclareMathOperator{\Ad}{Ad}
\DeclareMathOperator{\dvol}{dvol}

\DeclareMathOperator{\src}{src}
\DeclareMathOperator{\tgt}{tgt}
\DeclareMathOperator{\coker}{coker}
\DeclareMathOperator{\image}{image}

\newcommand{\MorRep}[0]{\Mor^{\mathrm{Rep}}}

\newcommand{\Ob}[1]{\mathrm{Ob}(#1)}
\newcommand{\Ar}[1]{\mathrm{Ar}(#1)}
\newcommand{\fp}[2]{\!\ {}_{#1}\!\times_{#2}}
\newcommand{\gpoid}[1]{#1_1\!\rightrightarrows\!#1_0}
\newcommand{\rd}[0]{\partial}
\newcommand{\pt}[0]{\mathrm{pt}}
\newcommand{\ul}[1]{\underline{#1}}
\newcommand{\pr}[0]{\mathrm{pr}}
\newcommand{\dual}[0]{{\scriptscriptstyle\vee}}

\newcommand{\pair}[1]{\langle#1\rangle}

\begin{document}

\maketitle

\begin{abstract}
 Making use of theory of differentiable stacks, we study symplectic vortex equations over a compact orbifold Riemann surface. We discuss the category of representable morphisms from a compact orbifold Riemann surface to a quotient stack. After that we define symplectic vortex equations over a compact orbifold Riemann surface. We also discuss the moduli space of solutions to the equations for linear actions of the circle group on the complex plane.
\end{abstract}

\section{Introduction} 
\label{sec:introduction}

The symplectic vortex equations (SVE for short) are introduced by Salamon and Mundet independently. Let $G$ be a compact Lie group and $\g$ its Lie algebra. We choose an invariant inner product on $\g$ to identify $\g$ with its dual space $\g^\dual$. Suppose that $G$ acts on a symplectic manifold $M$ in a Hamiltonian fashion with a moment map $\mu:M \to \g$. Let $P \to \Sigma$ be a principal $G$-bundle over a compact Riemann surface $\Sigma$. The SVE (over $\Sigma$) are defined by the PDEs
\[
 \brdA u = 0, \quad *F_A + \mu(u) = 0
\]
for $G$-equivariant map $u:P \to M$ and a connection $A$ on $P$. Here $\brdA$ is the anti-holomorphic part of the covariant derivative $\dA$, $F_A$ is the curvature of $A$, and $*$ is the Hodge $*$-operator. (We also have to fix a volume form on $\Sigma$ and an almost complex structure on $M$ compatible with the symplectic structure and the $G$-action.) Under some analytic assumptions, we can define Hamiltonian invariants as integrations over the moduli space of solutions to the SVE. The invariants are called symplectic vortex invariants.

The most important feature of the SVE is that we can obtain the equation of pseudo-holomorphic curve from $\Sigma$ to the Marsden--Weinstein quotient $\MW{M}G$ as a certain limit of the SVE. This limit is called the adiabatic limit, and this is the most important idea for an identity between Gromov--Witten invariants (with fixed marked points) of the smooth quotient $\MW{M}G$ and symplectic vortex invariants for the Hamiltonian $G$-space $M$ \cite{gaio05:_gromov_witten}. 

It is conjectured that the identity extends to Gromov--Witten invariants of the orbifold Marsden--Weinstein quotient $\MW{M}G$ (after a suitable extension of theory of SVE) \cite[\S 12.7]{mcduff04:_j}. Gromov--Witten invariants of closed symplectic orbifold are defined as integration over a certain compactification of the moduli space of representable pseudo-holomorphic maps from orbifold Riemann surfaces \cite{chen02:_orbif_gromov_witten}. Therefore it is natural to think about an extension of the SVE such that its adiabatic limit is the equation of representable pseudo-holomorphic maps. But if we put everything into orbifold settings (e.g.\! orbibundle, etc.) for SVE, then the theory becomes extremely complicated. 

To solve this problem, we make use of theory of differentiable stacks. Roughly speaking, the symplectic vortex equations over a Riemann surface $\Sigma$ can be thought as ``differential equations'' of maps from $\Sigma$ to the quotient stack $[M/G]$: a (representable) morphism of stacks $\phi:\Sigma \to [M/G]$ corresponds to the 2-cartesian square
\[
 \xymatrix@C=60pt@R=20pt{
 P \ar[r]^{u} \ar[d]_{\pi^P} & M \ar[d]^{\pi^M} \\
 \Sigma \ar[r]_{\phi} & [M/G],
 }
\]
where $u$ is a $G$-equivariant map and $\pi^P:P \to \Sigma$ is a principal $G$-bundle over $\Sigma$. 

In this paper, we see that we can replace $\Sigma$ with a stack corresponding to compact orbifold Riemann surface in the above diagram. One of the advantages of the replacement is that the upper part of the above diagram still belongs to the category of smooth manifolds. Finally we find that we can use the same PDEs for an extension of SVE after several observations on differentiable stacks.

This paper is organised as follows. In \S 2, we review briefly theory of differentiable stacks. In \S 3, we discuss the category of representable morphisms from an orbifold Riemann surface to a quotient stack. We show that the category is equivalent to a certain category similar to the category of principal $G$-bundles. In \S 4, we fix notations for an orbifold Riemann surface by using statements developed in \S 3, and we consider (representable) pseudo-holomorphic maps from a compact orbifold Riemann surface $\mathbf\Sigma$ to a (orbifold) Marsden--Weinstein quotient. After that we define SVE over $\mathbf\Sigma$ and discuss the moduli space of solutions to the SVE for the case when the circle group $S^1$ acts on the complex plane $\C$.

\paragraph{Notations} 
\label{notation-and-conventions}

We fix the following notations through this paper.

We denote by $\psi$ any right action map: if we have a right $G$-space $M$, then $\psi$ is defined by
\[
 \psi: M \times G \to M;\ (m,g) \mapsto \psi(m,g) = \psi_g(m) =m \cdot g = mg.
\]
Let $\g$ be the Lie algebra of a Lie group $G$. For a $G$-space $M$ and $\xi \in \g$, we define the infinitesimal action $\xi^M \in \Vect(M)$ by
\[
 \xi^M(m) = 
 \begin{cases}
  \left.\dfrac{d}{dt}\right|_{t=0} m \cdot \exp(t\xi) &
  (\text{if $M$ is a right $G$-space}), \\[6pt]
  \left.\dfrac{d}{dt}\right|_{t=0} \exp(-t\xi) \cdot m & 
  (\text{if $M$ is a left $G$-space})
 \end{cases}
 \quad (m \in M).
\]
Here $\Vect(M)$ is the space of vector fields on $M$. The above convention is the same as Bott-Tu \cite{bott01:_equiv_cartan} and opposite to Cieliebak--Gaio--Salamon \cite{cieliebak00:_j_hamil}. Set
\[
 \g^M(m) = \{ \xi^M(m) \in T_mM \mid \xi \in \g \} 
 \ \ \text{and}\ \
 \g^M = \{ \xi^M(m) \in TM \mid \xi \in \g,\ m \in M \}.
\]
If the $G$-action is locally free, then $\g^M$ is a subbundle of the tangent bundle $TM$.

For a (2-)fibred product $\X_1 \fp{F_1}{F_2} \X_2$, the $i$-th ($i=1,2$) projection is denoted by $\pr_i$.

\section{Differentiable stacks}

We review theory of differentiable stacks and geometric objects (e.g.\! differential forms) on stacks. First we recall the bicategory of Lie groupoids and bibundles. Secondly we introduce the 2-category of stacks over the smooth category and geometric objects on Deligne--Mumford stacks. Finally we describe an orbifold Riemann surface as a Deligne--Mumford stack.

\subsection{Lie groupoids and bibundles}

For details of Lie groupoids and bibundles, see Lerman \cite{lerman10:_orbif}, Metzler \cite{metzler03:_topol_smoot_stack}, Moerdijk \cite{moerdijk02:_orbif} and Moerdijk--Mr\v{c}un \cite{moerdijk03:_introd_lie}.

A \emph{groupoid} is a category whose arrows are all invertible. We write $\gpoid{X}$ for a groupoid whose class of objects and class of arrows are $X_0$ and $X_1$ respectively. The source map and the target map are always denoted by $\src$ and $\tgt$ respectively: for $a \in \Hom(x,y)$, $\src(a)=x$ and $\tgt(a)=y$. We also write $a:x \to y$ for such an arrow.

\begin{dfn}
 A groupoid $\gpoid{X}$ is called a \emph{Lie groupoid} if 
 \begin{enumerate}
  \item \squashup 
	both $X_0$ and $X_1$ are smooth manifolds, 
  \item \squashup 
	both the source and the target map are surjective submersions, and
  \item \squashup 
	the composition map 
	\[
	X_1 \fp{\src}{\tgt} X_1 \mapsto X_1; (a,b) \mapsto ab,
	\]
	the unit map $x \mapsto 1_x$ ($x \in X_0$) and the inverse map $a \mapsto a^{-1}$ ($a \in X_1$) are all smooth.
 \end{enumerate}
\end{dfn}


Let $\gpoid{X}$ be a Lie groupoid. For an object $x \in X_0$, the submanifold 
\[ 
(\gpoid{X})_x = \{ a \in X_1 \!\ |\!\ \src(a)=x=\tgt(a) \}
\]
of $X_1$ is a Lie group. The Lie group is called the \emph{stabiliser group} at $x$. 

Two objects $x, y \in X_0$ are said to lie in the same orbit ($x \sim y$) if there is an arrow $a:x \to y$. The quotient space of $X_0$ with respect to the equivalent relation $\sim$ is called the \emph{underlying space} of $\gpoid{X}$ and denoted by $X_0/X_1$. If $x \sim y$, then the stabiliser group at $x$ is naturally isomorphic to the stabiliser group at $y$. 

\begin{example}
 Let $G$ be a Lie group and $M$ a right $G$-space. The \emph{action groupoid} $M \rtimes G$ of the $G$-action on $M$ is a Lie groupoid $M \times G \rightrightarrows M$ whose structure maps are given as follows: $\src(m,g)=mg$, $\tgt(m,g)=m$, $(m,g)(mg,h)=(m,gh)$, $1_{m}=(m,1)$ and $(m,g)^{-1}=(mg,g^{-1})$ for $g,h \in G$ and $m \in M$. The stabiliser group at $m \in M$ is the ordinary stabiliser group at $m$ and the underlying space of the action groupoid $M \rtimes G$ is the quotient space $M/G$.
\end{example}

\begin{dfn} \hfill 

 \begin{itemize}
  \item \squashup
	A Lie groupoid $\gpoid{X}$ is said to be \emph{proper} if 
	$(\src,\tgt) : X_1 \mapsto X_0 \times X_0$ is a proper map.
  \item \squashup
	A Lie groupoid is called an \emph{\'etale groupoid} if the source and the target maps are \'etale (locally diffeomorphic).
 \end{itemize}
\end{dfn}

Next we recall bibundles. Regarding a bibundle as an arrow between Lie groupoids, we obtain the bicategory of Lie groupoids and bibunbles.

\begin{dfn}
 A \emph{right action} of a Lie groupoid $\gpoid{Y}$ on a manifold $P$ consists of two maps $\alpha:P \to Y_0$ and
 \[
 \mu:P \fp{\alpha}{\tgt} Y_1 \to P; \ (p,a) \mapsto pa
 \]
 satisfying that $\alpha(pa)=\src(a)$, $p1_{\alpha(p)}=p$ and $p(ab)=(pa)b$ for any $(a,b) \in Y_1 \fp{\src}{\tgt} Y_1$ and $p \in P$ with $\alpha(p)=\tgt(a)$. The map $\alpha$ is called the \emph{anchor map}. 
%
\end{dfn}

A left action of a Lie groupoid can be defined in a similar way.

\begin{dfn}
 Let $\gpoid{Y}$ be a Lie groupoid. A \emph{principal $(\gpoid{Y})$-bundle} over a manifold $B$ is a smooth map $\pi:P \to B$ with a right action of $\gpoid{Y}$ on $P$ satisfying the following conditions.
\begin{enumerate}
 \item \squashup 
       The map $\pi$ is $(\gpoid{Y})$-invariant, i.e.\!
       $\pi(pa)=\pi(p)$ for any $(p,a) \in P \fp{\alpha}{\tgt} Y_1$ .
 \item \squashup The map $\pi$ is a surjective submersion.
 \item \squashup The map 
       \[
       P \fp{\alpha}{\tgt} Y_1 \to P \fp{\pi}{\pi} P;\ 
       (p,a) \mapsto (p,pa)
       \]
       is a diffeomorphism. Here $\alpha$ is the anchor map for the action.
\end{enumerate}
\end{dfn}


\begin{dfn}
 Let $\gpoid{X}$ and $\gpoid{Y}$ be two Lie groupoids. A \emph{bibundle} (or Hilsum--Skandalis morphism) $f = (\alpha_L,P,\alpha_R)$ is a manifold $P$ equipped with 
 \begin{itemize}
  \item \squashup
	a left action of $\gpoid{X}$ with an anchor map $\alpha_L$ and 
  \item \squashup
	a right action of $\gpoid{Y}$ with an anchor map $\alpha_R$
 \end{itemize}
\squashup
satisfying following conditions.
\begin{enumerate}
 \item \squashup
       The map $\alpha_L:P \to X_0$ is $(\gpoid{Y})$-invariant and $\alpha_R:P \to Y_0$ is $(\gpoid{X})$-invariant.
 \item \squashup
       The actions of $\gpoid{X}$ and $\gpoid{Y}$ are compatible, i.e. $(ap)b=a(pb)$ for any $a \in X_1$, $p \in P$ and $b \in Y_1$ with $\src(a)=\alpha_L(p)$ and $\alpha_R(p)=\tgt(b)$.
 \item \squashup
       The map $\alpha_L:P \to X_0$ is a principal $(\gpoid{Y})$-bundle over $X_0$. 
\end{enumerate}
 
 Given two bibundles $f=(\alpha_L,P,\alpha_R)$ and $f'=(\alpha'_L,P',\alpha'_R)$, a \emph{2-isomorphism} from $f$ to $f'$ is a diffeomorphism $\phi:P \to P'$ which commutes with both $(\gpoid{X})$- and $(\gpoid{Y})$-actions.
\end{dfn}



The composition of bibundles is not strictly associative, but associative up to a (canonical) 2-isomorphism. Therefore Lie groupoids and bibundles do not form a category, but form a bicategory. We denote by $\Bi$ the bicategory of Lie groupoids and bibundles.

Any bibundle induces a continuous map between underlying spaces and a group homomorphism between stabiliser groups. Suppose that we have a bibundle $f=(\alpha_L,P,\alpha_R)$ from $\gpoid{X}$ to $\gpoid{Y}$. For any $p,p' \in P$ and $a \in X_1$ satisfying $\src a=\alpha_R(p)$ and $\tgt a=\alpha_R(p')$, there is a unique $b \in Y_1$ such that $ap=pb$, $\tgt b = \alpha_R(p')$ and $\src b = \alpha_R(p)$. This fact guarantees that the map 
\[
F^f: X_0/X_1 \to Y_0/Y_1;\ [\alpha_R(p)] \mapsto [ \alpha_L(p)] \quad (p \in P)
\]
is a well-defined continuous map and the map 
\[
 (\gpoid{X})_{\alpha_R(p)} \to (\gpoid{Y})_{\alpha_L(p)};\ a \mapsto b
\]
is a group homomorphism. (Two isomorphic bibundles induces the same continuous map and the group homomorphism.) If two groupoids $\gpoid{X}$ and $\gpoid{Y}$ are equivalent in the bicategory $\Bi$, then the underlying spaces $X_0/X_1$ and $Y_0/Y_1$ are homeomorphic and the group structures of stabiliser groups at the points corresponding under the homeomorphism are isomorphic. Therefore topological structures on the underlying space and stabiliser groups are invariants of the category $\Bi$.

\subsection{Differentiable stacks}

For details of stacks, see Behrend--Xu \cite{behrend11:_differ}, Heinloth \cite{heinloth05:_notes}, Metzler \cite{metzler03:_topol_smoot_stack} and Lerman \cite{lerman10:_orbif}.

Roughly speaking, a stack is a category $\X$ equipped with a functor from $\X$ to a base site satisfying some conditions including a so-called ``gluing condition''. Here a \emph{site} is a category equipped with a Grothendieck topology. Theory of stacks can be discussed for several base sites, but we only use the category $\Diff$ of smooth manifolds and smooth maps as a base site in this paper.

We define a Grothendieck topology on $\Diff$ as follows. For $U \in \Diff$, a family $\{f_i:U_i \to U\}_i$ of smooth maps to $U$ is called a \emph{covering family} of $U$ if $f_i:U_i \to U$ is a local diffeomorphism for each $i$ and the total map $\bigsqcup U_i \to U$ is surjective. The function $K$ assigning to each object $U \in \Diff$ the collection $K(U)$ of covering families define a basis of Grothendieck topology (cf Metzler \cite[Definition 5]{metzler03:_topol_smoot_stack}). The basis of Grothendieck topology generates a Grothendieck topology on the category $\Diff$. 

\begin{remark}
For general definition of Grothendieck topology see Metzler \cite{metzler03:_topol_smoot_stack} and Vistoli \cite{vistoli08:_notes_groth}. Note that a basis for a Grothendieck topology (pretopology) is called a Grothendieck topology in Vistoli \cite[Definition 2.24]{vistoli08:_notes_groth}. We follow Behrend--Xu \cite{behrend11:_differ} for definition of covering families in $\Diff$ and the same definition is used in Lerman \cite[Remark 2.17]{lerman10:_orbif} and Lerman--Malkin \cite{lerman12:_hamil_delig_mumfor}. Metzler uses a slightly different definition for covering family, but the Grothendieck topology is the same as our Grothendieck topology.
\end{remark}

\begin{dfn}
 \label{dfn:Category Fibred in Groupoids}
 A \emph{category fibred in groupoids (over $\Diff$)} is a category $\X$ equipped with a functor $F_{\X}: \X \to \Diff$ satisfying the following conditions.
 \begin{enumerate}
  \item \squashup
	For any $f: V \to U$ in $\Diff$ and any $x \in \X$ with $F_{\X}(x)=U$, there is an arrow $a:y \to x$ in $\X$ such that $F_{\X}(a)=f$.
  \item \squashup
	Suppose we have two arrows $a_1:y_1 \to x$ and $a_2:y_2 \to x$. For any smooth map $f:F_{\X}(y_1) \to F_{\X}(y_2)$ with $F_{\X}(a_2) \circ f = F_{\X}(a_1)$ there is a unique arrow $b:y_1 \to y_2$ such that $a_2 b = a_1$ and $F_{\X}(b) = f$. 
 \end{enumerate}
 \squashup
 The functor $F_{\X}: \X \to \Diff$ is called the \emph{base functor of the category fibred in groupoids}.
\end{dfn}

Omitting the base functor $F_{\X}$, we often say that $\X$ is a category fibred in groupoids. Moreover $F_{\X}$ is described as ``$\ul{\ \ }$''. Namely $\ul{x}=F_{\X}(x)$ for $x \in \X$ and $\ul{a}=F_{\X}(a)$ for an arrow $a$ in $\X$.

The collection of all categories fibred in groupoids form a 2-category:
\begin{dfn}
 A \emph{morphism of categories fibred in groupoids} from $\X$ to $\Y$ is a functor $\phi:\X \to \Y$ which commutes with the base functors. 

 For two morphisms of categories fibred in groupoids $\phi_1,\phi_2:\X \to \Y$, a \emph{2-morphism of categories fibred in groupoids} is a natural transformation $\alpha:\phi_1 \to \phi_2$ such that the horizontal composition $F_{\Y} \ast \alpha$ is the identity transformation of $F_{\X}$.
\end{dfn}

The \emph{fibre of $\X$ over $U \in \Diff$} is the groupoid $\X_U$ with objects $\{ x \in \X\ |\ \underline{x} = U \}$ and arrows $\{ \text{arrow } a \text{ in } \X\ |\ \ul{a}=\id_U \}$. 

Given a category fibred in groupoids $\X$, for every object $x \in \X$ and every $f:V \to \ul{x}$ in $\Diff$ we choose an arrow $a:y \to x$ in $\X$ such that $\ul{a}=f$ (cf Definition \ref{dfn:Category Fibred in Groupoids}). The object $y$ is called the \emph{pullback of the object $x$ via $f$} and denoted by $f^*x$. 

Let $a:x \to y$ be an arrow in a fibre $\X_U$ and $f:V \to U$ a smooth map. Then we have two arrows $b_x:f^*x \to x$ and $b_y:f^*y \to y$ which we have chosen. Then by Definition \ref{dfn:Category Fibred in Groupoids}, there is a unique arrow $\tilde{a}:f^*x \to f^*y$ such that $\tilde{a}$ belongs to $\X_V$. The arrow $\tilde{a}$ is called the \emph{pullback of the arrow $a$ via $f$} and denoted by $f^*a$.

We can assign the descent category $\X_{\{U_i \to U\}}$ to each covering family $\{f_i:U_i \to U\}$ and there is a natural functor from the fibre $\X_U$ to the descent category $\X_{\{U_i \to U\}}$. The category fibred in groupoids is called a \emph{stack} (over $\Diff$) if for any covering family the natural functor is an equivalence of categories. 

The collection of all stacks (over $\Diff$) form a 2-subcategory $\StDiff$ of the 2-category of categories fibred in groupoids. We regard two stacks as the same stack if they are equivalent. 

\begin{remark}
 Some authors say two stacks which are equivalent are isomorphic. But we follow the ordinary terminology of theory of 2-categories.
\end{remark}

 Let $G$ be a Lie group. For a right $G$-space $M$,  we can define a stack $[M/G]$ as follows. An object of $[M/G]$ is a pair $(\pi,\epsilon)$ of a principal $G$-bundle $\pi:P \to U$ over a manifold $U$ and a $G$-equivariant map $\epsilon:P \to M$. The object $(\pi,\epsilon)$ is often denoted by $U \xleftarrow{\pi} P \xrightarrow{\epsilon} M$. An arrow from $V \xleftarrow{\pi^Q} Q \xrightarrow{\epsilon^Q} M$ to $U \xleftarrow{\pi^P} P \xrightarrow{\epsilon^P} M$ is a pair $(f,\tilde{f})$ of smooth map $f:V \to U$ and a $G$-equivariant map $\tilde{f}:Q \to P$ which makes the following diagram commutative:
\[
 \xymatrix@C=60pt@R=4pt{
 V \ar[dd]_{f}& Q \ar[dd]_{\tilde{f}}\ar[l]_{\pi^Q}\ar[rd]^{\epsilon^Q} & \\
   &   & M. \\
 U & P \ar[l]^{\pi^P}\ar[ru]_{\epsilon^P} & 
 }
\]
The category $[M/G]$ is a category fibred in groupoid with the base functor
\[
 [M/G] \to \Diff; \ 
 \begin{cases}
  (U\xleftarrow{\pi}P\xrightarrow{\epsilon}M) \mapsto U & \text{for objects},\\
  (f,\tilde{f}) \to f & \text{for arrows}.
 \end{cases}
\]
For any object $U\xleftarrow{\pi}P\xrightarrow{\epsilon}M$ over $U$ and any smooth map $f:U \to V$, the diagram $V \xleftarrow{\pr_1} V \fp{f}{\pi} P \xrightarrow{\epsilon\circ\pr_2} M$ gives an object of over $V$ and we have an arrow in $[M/G]$.
\[
 \xymatrix@C=60pt@R=4pt{
 V \ar[dd]_{f}& V \fp{f}{\pi} P \ar[dd]_{\pr_2}\ar[l]_{\pr_1}\ar[rd]^{\epsilon\circ\pr_2} & \\
   &   & M. \\
 U & P \ar[l]^{\pi}\ar[ru]_{\epsilon} & 
 }
\]
We choose the object $V \xleftarrow{\pr_1} V \fp{f}{\pi} P \xrightarrow{\epsilon\circ\pr_2} M$ as the pullback of $U\xleftarrow{\pi}P\xrightarrow{\epsilon}M$ via $f$. We can show that the category fibred in groupoids $[M/G]$ is a stack. The stack is called the \emph{quotient stack} associated to the $G$-space.

An arbitrary manifold $M$ can be considered as a quotient stack $[M/\{1\}]$: An object is a smooth map $f$ whose target is $M$ and an arrow from $g:V \to M$ to $f:U \to M$ is a smooth map $a:V \to U$ with $f \circ a = g$. Thanks to the 2-Yoneda embedding, the category $\Diff$ can be embedded into the 2-category of stacks $\StDiff$. We identify every manifold $M$ with the stack $[M/\{1\}]$. A stack $\X$ is said to be \emph{representable} if there is a manifold equivalent to the stack $\X$.

In $\StDiff$ we can always take a $2$-fibre product. 

\begin{dfn}
 Let $\phi_1:\Y_1 \to \X$ and $\phi_2:\Y_2 \to \X$ be two morphisms of stacks. The \emph{(2-)fibre product} $\Y_1 \fp{\phi_1}{\phi_2} \Y_2$ is a stack defined as follows. The class of objects is 
 \[
 \{ (y_1,y_2,a) \mid \ y_1 \in \Y_1,\ y_2 \in \Y_2,\ 
 \ul{y} = \ul{y'} = U,\ a:\phi(y) \to \phi'(y_2) \text{ in } \X_U
 \}.
 \]
An arrow from $(y_1,y_2,a)$ to $(y_1',y_2',a')$ is a pair of arrows $(b_1:y_1 \to y_1',b_2:y_2 \to y_2')$ satisfying $\ul{b_1}=\ul{b_2}$ and $a' \circ \phi_1(b_1) = \phi_2(b_2) \circ a$. The base functor $\Y_1 \fp{\phi_1}{\phi_2} \Y_2 \to \Diff$ is given by $\ul{(y_1,y_2,a)}=\ul{y_1}$ for objects and $\ul{(b_1,b_2)}=\ul{b_1}$ for arrows. In particular $\Y_1 \times \Y_2$ is defined as the 2-fibre product $\Y_1 \fp{\phi_1}{\phi_2} \Y_2$ for $\phi_1:\Y_1 \to \pt$ and $\phi_2:\Y_2 \to \pt$. 
\end{dfn}

\begin{dfn}
 A morphism of stacks $\phi:\Y \to \X$ is said to be \emph{representable} if for any morphism of stack $F$ from a manifold $U$ to $\X$ the fibre product $U \fp{F}{\phi} \Y$ is representable.
\end{dfn}

\begin{dfn}
 A representable morphism $\phi$ from a manifold $X$ to a stack $\X$ is called an \emph{atlas} (resp. \emph{\'etale atlas}) of $\X$ if for any morphism $F$ from a manifold $U$ to the stack $\X$ the projection map from the representable stack $U \fp{F}{\phi} X$ to $U$ is a surjective submersion (resp. surjective local diffeomorphism). 
\end{dfn}

Let $\phi:X_0 \to \X$ be an atlas of a stack $\X$. Choosing a manifold $X_1$ which is equivalent to the representable stack $X_0 \fp{\phi}{\phi} X_0$, we obtain a 2-cartesian diagram:
\[
 \xymatrix@C=60pt@R=20pt{
 X_1 \ar[r]^{s}\ar[d]_{t} & X_0 \ar[d]^{\phi} \\
 X_0 \ar[r]_{\phi}& \X.
 }
\]
We can see that $\gpoid{X}$ is a Lie groupoid whose source map and target map are $s$ and $t$, respectively. The Lie groupoid $\gpoid{X}$ is called the \emph{presentation} of $\X$ associated to the atlas. 

Two presentations of the same stack associated to two atlases are equivalent in $\Bi$, therefore any invariant of the category $\Bi$ (e.g. the underlying space, the stabiliser groups, etc.) is also an invariant of differentiable stacks. See Lerman \cite{lerman10:_orbif} for more details. 

\begin{dfn}
A a stack $\X$ is called a \emph{differentiable stack} (resp \emph{Deligne--Mumford stack}) if $\X$ has an atlas (resp \'etale atlas) and the presentation associated to the atlas is proper.
\end{dfn}

\begin{dfn}
 Let $\gpoid{X}$ be the presentation of $\X$. Suppose that both $X_1$ and $X_0$ have constant dimensions. We define the \emph{dimension} of $\X$ by $2\dim X_0 - \dim X_1$.
\end{dfn}

\begin{dfn}
 Let $\gpoid{X}$ be the presentation of $\X$ associated to an atlas.
 \begin{itemize}
  \item \squashup
	A differentiable stack $\X$ is said to be \emph{compact} if so is $X_0/X_1$.
  \item \squashup
	A differentiable stack $\X$ is said to be \emph{connected}
	if so is $X_0/X_1$.
 \end{itemize}
\end{dfn}

\begin{dfn} 
 \label{dfn:type-R}
 A stack $\X$ is said to be \emph{type-R} if 
 \begin{enumerate}
  \item \squashup
	the stack $\X$ is a compact and connected Deligne--Mumford stack, and 
  \item \squashup
	a generic stabiliser group of $\X$ is trivial.
 \end{enumerate} 
\end{dfn}

\begin{example}
 Let $G$ be a compact Lie group and $M$ a (right) $G$-space. The quotient stack $[M/G]$ always has an atlas $\pi^M:M \to [M/G]$ called the \emph{natural projection}. For objects, $\pi^M(m:U\to M) = (U \xleftarrow{\pr_1} U \times G \xrightarrow{\epsilon}M)$, where $\epsilon(u,g)=m(u)g$. Therefore the quotient stack is differentiable. The action groupoid $M \rtimes G$ is a presentation of the atlas. If the $G$-action is locally free, then the quotient stack is Deligne--Mumford. (This does not mean that the natural projection is an \'etale atlas.)
\end{example}

\subsection{Differential forms and vector fields over stacks}
\label{sec:diff-forms-and-vector-fields-on-stack}

We describe geometric objects (e.g. differential forms, vector fields, etc) on Deligne--Mumford stacks using sheaves over the stack. The general definition of a sheaves over stacks can be found in Behrend--Xu \cite{behrend11:_differ} and Metzler \cite{metzler03:_topol_smoot_stack}. Kashiwara and Schapira \cite{kashiwara06:_categ} explain a general theory of sheaves over sites. Note that a Grothendieck topology on a stack can be induced by the Grothendieck topology on the category $\Diff$. 

\paragraph{Differential forms over a differentiable stack \cite{behrend04:_cohom,behrend11:_differ}}
We define the sheaf $\Omega^k_{\X}$ of differential $k$-forms on a differentiable stack $\X$ as follows. For an object $x \in \X$ over $U$, $\Omega^k_{\X}$ is defined by the space $\Omega^k(U)$ of $k$-forms on $U$. For an arrow $a:x \to y$ in $\X$ with $\ul{a}:U \to V$, we assign the pullback map $\ul{a}^*:\Omega^k(V) \to \Omega^k(U)$ to $\Omega^k_{\X}(y) \to \Omega^k_{\X}(x)$. This presheaf $\Omega^k_\X: \X \to (\R\text{-}\mathrm{Vect})$ satisfies the conditions to be a sheaf over $\X$. Since the exterior derivative $d$ commutes with pullbacks, $d$ makes the abelian sheaves $\Omega_\X^*$ a complex. The complex $\Omega_\X^*$ is called the (big) de Rham complex of $\X$. The de Rham cohomology of $\X$ is defined as the hypercohomology of $\X$ with values in $\Omega_\X^*$. 

Using the presentation $\gpoid{X}$ associated to an atlas of $\X$, we can calculate the de Rham cohomology $\HDR^*(\X)$ more explicitly. The set of (global) $k$-forms on $\X$ is given by 
\[
 \Omega^k(\X) = \{ \eta \in \Omega^k(X_0) \mid \src^*\eta=\tgt^*\eta \}.
\]
The de Rham cohomology is isomorphic to the cohomology of the complex $(\Omega^k(\X),d)$.

\begin{remark}
 In general, the set of global sections of a sheaf $\F$ over a stack $\X$ is defined by applying the global section functor $\Gamma$ to $\F$. If $\X$ is differentiable, the set is canonically isomorphic to the equaliser of the two restriction maps $\F(X_1)\rightrightarrows\F(X_0)$, where $\gpoid{X}$ is a presentation associated to an atlas $X_0 \to \X$. 
\end{remark}

\begin{remark}
 For any vector space $V$, we can define the sheaf of differential $k$-forms on a differentiable stack $\X$ in a similar way.
\end{remark}

\paragraph{Vector fields over a Deligne--Mumford stack \cite{lerman12:_hamil_delig_mumfor}}
If we have a Deligne--Mumford stack $\X$, then we can define the tangent sheaf $\Ts_\X$ as follows. Let $\gpoid{Z}$ be the presentation of $\X$ associated to the \'etale atlas $p: Z_0 \to \X$. For $x \in \X_U$ we have the following cartesian diagrams:
\[
 \xymatrix@R=20pt@C=35pt{
 V_1 \ar@<0.5ex>[r] \ar@<-0.5ex>[r] \ar[d]^{f_1} & V_0 \ar[r]^{q} \ar[d]^{f_0} & U \ar[d]^{x}\\
 Z_1 \ar@<0.5ex>[r] \ar@<-0.5ex>[r] & Z_0 \ar[r]_{p} & \X
 }
\]
Note that $\gpoid{V}$ is an \'etale groupoid and a presentation of $U$. Since $\gpoid{Z}$ is \'etale, $\src^*T{Z_0}=\tgt^*T{Z_0}=T{Z_1}$. According to descent theory, there is a vector bundle $E \to U$ (which is unique up to isomorphisms) such that the pullback $q^*E$ is isomorphic to $f_0^*TZ_0$. Then define $\Ts_\X(x)$ as the set of global sections of the bundle $E \to U$. For an arrow $a:x \to y$, $\Ts_\X(a):\Ts_\X(y) \to \Ts_\X(x)$ is naturally defined. Moreover the tangent sheaf $\Ts_\X$ is independent of the choice of the \'etale atlas. The set of global sections of the tangent sheaf is denoted by $\Vect(\X)$ and an element of $\Vect(\X)$ is called a vector field on $\X$.

\begin{prop}
 [Lerman--Malkin \cite{lerman12:_hamil_delig_mumfor}]
 \label{prop:tangent-sheaf-and-vector-fields}
 Let $\gpoid{X}$ be a presentation of an atlas $p:X_0 \to \X$ of a Deligne--Mumford stack $\X$. 
 \begin{enumerate}
  \item Let $A$ be the pullback bundle of $\ker(d\src) \to X_1$ along the unit map $X_0 \to X_1$. Then the map $d\tgt:A \to TX_0$ is an injective bundle map. (Thus we regard $A$ as a subbundle of $TX_0$.)
  \item \squashup
	The (small) sheaf over $X_0$ induced by the tangent sheaf $\Ts_\X$ is the sheaf of sections of the bundle $TX_0/A$.
  \item \squashup
	The (small) sheaf over $X_1$ induced by the tangent sheaf $\Ts_\X$ is the sheaf of sections of the bundle $TX_1/(\ker d\src + \ker d\tgt)$.
  \item \squashup
	The set of vector fields $\Vect(\X)$ is explicitly given by the following quotient vector space:
	\[
	\V/\{ (v_1,v_0) \in \V \mid v_1 \in \ker(d\src)+\ker(d\tgt) \}.
	\]
	Here $\V = \{(v_1, v_0) \in \Vect(X_1)\times\Vect(X_0) \mid d\src \circ v_1 = v_0 \circ \src,\ d\tgt \circ v_1 = v_0 \circ \tgt\}$ and $\Vect(X_i)$ is the space of the (ordinary) vector fields on $X_i\ (i=0,1)$. In particular, if the atlas $p:X_0 \to \X$ is \'etale, then $\Vect(\X)$ is isomorphic to $\V$. 
 \end{enumerate}
\end{prop}

\paragraph{Symplectic forms \cite{lerman12:_hamil_delig_mumfor}}
Let $\gpoid{X}$ be a presentation of a Deligne--Mumford stack $\X$ associated to an atlas $p:X_0 \to \X$. For a vector field $(v_1,v_0) \in \Vect(\X)$, we can define the interior product 
\[
\iota(v_1,v_0):\Omega^k(\X)\to\Omega^{k-1}(\X);\
\eta \mapsto \iota(v_0)\eta.
\]

A 2-form $\omega \in \Omega^2(\X)$ on $\X$ is said to be nondegenerate if the map
\[
 \Vect(\X) \to \Omega^1(\X);\ 
 (v_1,v_0) \mapsto \iota(v_1,v_0)\omega
\]
is a linear isomorphism. This is equivalent to the condition that $\ker\omega=A$. A closed nondegenerate 2-form $\omega$ on $\X$ is called a symplectic form on $\X$ and a symplectic Deligne--Mumford stack is a pair $(\X,\omega)$ of a Deligne--Mumford stack $\X$ and a symplectic form $\omega$ on $\X$. 

\paragraph{Orientations and almost complex structures}

Let $\X$ be a differentiable stack of dimension $n$. The stack $\X$ is said to be orientable if $\Omega^n(\X)$ is isomorphic to $\O(\X)$ as an $\O(\X)$-module. For an orientable stack $\X$ we define 
\[
 \Omega^n(\X)_\mathrm{gen}
 = \{ \omega \in \Omega^n(\X) \mid \O(\X)\omega=\Omega^n(\X) \}.
\]
and we denote by $\O(\X)_+$ the space of (global) positive functions. An orientation of $\X$ is a choice of an element of the quotient set $\Omega^n(\X)_\mathrm{gen}/\O(\X)_+$. For a Deligne--Mumford stack $\X$, an orientation is nothing but a pair of orientations of $Z_0$ and $Z_1$ compatible with the source and target maps.

Let $\X$ be a Deligne--Mumford stack. An almost complex structure is an endomorphism $J:\Ts_\X\to\Ts_\X$ of $\O_\X$-modules satisfying $J^2=-\id_{\Ts_\X}$. For an atlas $p: X_0 \to \X$, an almost complex structure $J$ on $\X$ induces a complex structure of the vector bundle $TX_0/A$ compatible with the source and the target map. In particular, for an \'etale atlas $\zeta:Z_0\to\X$, an almost complex structure $J$ induces almost complex structures on $Z_0$ and $Z_1$ which are compatible with the source and target maps. Here $\gpoid{Z}$ is an presentation associated to the atlas $\zeta:Z_0\to\X$. In particular, an almost complex structure induces a canonical orientation on $\X$.

\paragraph{Integral over a Deligne--Mumford stack \cite{behrend04:_cohom}}
A partition of unity of a proper \'etale groupoid $\gpoid{Z}$ is a smooth function $\rho$ on $Z_0$ such that $\src^*\rho$ has proper support with respect to $\tgt:Z_1 \to Z_0$ and $\tgt_!\src^*\rho = 1$. Not all proper \'etale groupoid has a partition of unity, but for any proper \'etale groupoid $\gpoid{Z}$, we can find a proper \'etale groupoid $\gpoid{Z'}$ which is equivalent to $\gpoid{Z}$ and has a partition of unity. 

Let $\X$ be an oriented and compact Deligne--Mumford stack with an \'etale atlas $\zeta:Z_0 \to \X$ and $\gpoid{Z}$ a presentation associated to the atlas. If the presentation $\gpoid{Z}$ has a partition of unity $\rho:Z_0 \to \R$, then we can define an integral
\[
 \int_\X : \Omega^n(\X) \to \R;\ \ \int_\X \omega  = \int_{Z_0} \rho\!\ \omega. 
\]
Here $n=\dim \X = \dim Z_0$. This gives rise to a linear map $\H^n(\X)\to\R$ as usual. Moreover the Poincar\'e duality holds: The pairing 
\[
 \HDR^p(\X) \otimes \HDR^{n-p}(\X) \to \R;\ 
 ([\omega],[\eta]) \mapsto \int_\X \omega\wedge\eta
\]
is nondegenerate. 

\subsection{An orbifold Riemann surface as a stack}
\label{sec:orbifold-Riemann-surface}

In this paper we consider an orbifold Riemann surface as a Deligne--Mumford stack. We describe how to construct the stack in this subsection. We basically follow terminologies of Moerdijk--Pronk \cite{moerdijk97:_orbif,moerdijk99:_simpl} for orbifolds.

A compact orbifold Riemann surface is a quadruple $\mathbf{\Sigma}=(\Sigma,j,\mathbf{z},\mathbf{m})$ of a closed Riemann surface $\Sigma$, a complex structure $j$ on $\Sigma$, a $k$-tuple $\mathbf{z}=(z_1,\dots,z_k)$ of distinct points on $\Sigma$ and a $k$-tuple $\mathbf{m}=(m_1,\dots,m_k)$ of positive integers. We define the \emph{multiplicity} $m$ of $p \in \Sigma$ as follows: 
\[
  m = 
  \begin{cases}
   m_i  & \text{if}\ p=z_i\ \text{for some}\ i,\\
   1 & \text{else}.
  \end{cases}
\]
A point $p \in \Sigma$ is said to be \emph{smooth} if its multiplicity is $1$. 

First we recall an orbifold atlas for the compact orbifold Riemann surface $\mathbf{\Sigma}$. For a point $p$ of multiplicity $m$, we choose a complex coordinate system $w_p:V_p \to \C$ centred at $p$ so that $V_p \setminus \{p\}$ consists only of smooth points. We also choose an open disc $U_p \subset \C$ centred at the origin and a holomorphic map $\pi_p:U_p \to \Sigma$ satisfying $\pi_p(0)=p$, $\pi_p(U_p)=V_p$ and $(w_p\circ\pi_p)(z)=z^m$. Here $z$ is the standard coordinate function on $\C$. Then the group 
\[
 G_p = \left\{ e^{\frac{2\pi\i k}{m}} \in \C \!\ \middle|\!\  k=0,\dots,m-1 \right\} \cong \Z/m\Z
\]
acts on $U_p$ as multiplication and the map $\pi_p$ induces a homeomorphism from $U_p/G_p$ to $V_p$. Then $(U_p,G_p,\pi_p)$ is an orbifold chart for $\pi_p(U_p)$ and the family $\U = \{(U_p,G_p,\pi_p)\}_{p \in \Sigma}$ is an orbifold atlas on $\Sigma$.

We consider the tangent bundle $T\mathbf{\Sigma}$. This is a complex orbibundle over $\mathbf\Sigma$. Choosing a metric compatible with the complex structure, we obtain the $S^1$-bundle $P$ consisting of unit tangent vectors. It is easy to see that the total space of the bundle is a smooth manifold. Moreover the circle group $S^1$ acts locally freely on $P$ as complex multiplication (on the right). 

Since the orbifold $\mathbf\Sigma$ can be obtained as the quotient of $P$ with respect to the $S^1$-action, we regard the quotient stack $\X = [P/S^1]$ as the orbifold Riemann surface $\mathbf\Sigma$. By the construction of $\X$, the stack $\X$ is type-R. 

By Proposition \ref{prop:tangent-sheaf-and-vector-fields}, the sheaf over $P$ induced by the tangent sheaf is the sheaf of sections of $N=TP/\g^P$. Here $\g$ is the Lie algebra of $S^1$. The bundle $N$ is an $S^1$-equivariant vector bundle, and $N/S^1 \to P/S^1$ is naturally identified with the tangent bundle $T\mathbf\Sigma$. Therefore the vector bundle $N$ is equipped with a complex structure compatible with the $S^1$-action. The complex structure gives the almost complex structure on $\X$ which will be denoted by $j$ as well.

\section{The category of representable morphisms}

In this section, we discuss the category of representable morphisms of stacks from a (type-R) stack to a quotient stack. The category of (representable) morphisms from a manifold to a quotient stack is a prototype. Let $M$ be a $G$-space. By the 2-Yoneda lemma and the definition of quotient stacks, any (representable) morphism of stacks $\phi$ from a manifold $U$ to the quotient stack $[M/G]$ corresponds to the 2-cartesian square
\[
 \xymatrix@C=60pt@R=20pt{
 P \ar[r]^{u} \ar[d]_{\pi^P} & M \ar[d]^{\pi^M} \\
 U \ar[r]_{\phi} & [M/G].
 }
\]
Here $\pi^P:P \to U$ is a principal $G$-bundle and $u:P \to M$ is a $G$-equivariant map. To replace $U$ with a (type-R) stack $\X$, we discuss a certain category $\P^G(\X)$ which is similar to the category of principal $G$-bundles. After that we show that the group of automorphisms of the category for the type-R base $\X$ is the same as the ordinary gauge group action. Finally we discuss the category of representable morphisms from $\X$ to the quotient stack $[M/G]$.

\subsection{The category $\P^G(\X)$}

For a stack $\X$ and a compact Lie group $G$, we construct a category $\P^G(\X)$ as follows. 

\begin{dfn}
 An object of $\P^G(\X)$ is a pair of a morphism $\pi^P:P \to \X$ from a right $G$-space $P$ to $\X$ and a 2-morphisms of stacks $\sigma: \pi^P \circ \pr_1 \to \pi^P \circ \psi$ 
 \[
 \xymatrix@C=15pt@R=10pt{
 P \times G \ar[rrr]^{\psi}\ar[dd]_{\pr_1} & & & P \ar[dd]^{\pi^P} \\
 & & & \\
 P \ar[rrr]_{\pi^P} & & \ar@{=>}[ru]^{\sigma} & \X.
 }
 \]
 satisfying the following conditions. 
 \begin{enumerate}
  \item \squashup
	The morphism $\pi^P: P \to \X$ is an atlas. 
  \item \squashup
	The morphism of stacks
	\[
	P \times G \to P \fp{\pi^P}{\pi^P} P;\
	\begin{cases}
	 (p,g) \to (p,pg,\sigma(p,g)) & \text{for objects,}\\
	 f \to (f,f) & \text{for arrows}
	\end{cases}
	\]
	is an equivalence. (Therefore the above square is 2-cartesian.)
  \item \squashup 
	For any $p \in P$, $g, h \in G$ with $\ul{p}=\ul{g}=\ul{h}$, the following identity holds.
	\begin{equation}
	 \label{eqn:2-morph-in-group-action}
	 \sigma(p,gh)=\sigma(pg,h)\sigma(p,g).
	\end{equation}
 \end{enumerate}
\end{dfn}

\begin{remark}
 \label{remark:category-PG-and-quotient-stack}
 If $(\pi:P \to \X,\sigma)$ is an object of $\P^G(\X)$, the stack $\X$ is equivalent to the quotient stack $[P/G]$. On the other hand, for any quotient stack $[P/G]$ we can naturally obtain an object $(\pi^P:P\to[P/G], \sigma^P)$ of $\P^G([P/G])$. (For $(p,g)\in (P \times G)_U$, $\sigma^P(p,g): U\times G \to U \times G$ is defined by $\sigma^P(p,g)(u,a)=(u,g(u)^{-1}a)$.) The point of the definition of $\P^G(\X)$ is that we can stick with a fixed base stack $\X$. 
\end{remark}


\begin{dfn}
 An arrow between two objects $(\pi:P \to \X,\sigma)$ and $(\pi':P' \to \X,\sigma')$ of $\P^G(\X)$ is a pair of a $G$-equivariant diffeomorphism $\phi:P \to P'$ and a 2-morphism of stacks $\tau:\pi \to \pi'\circ\phi$ such that for any $(p,g) \in P \times G$ the identity $\tau(pg)\sigma(p,g)=\sigma'(\phi(p),g)\tau(p)$ hold.
 \[
 \xymatrix{
 P \ar[rd]_{\pi}\ar[rr]^{\phi}& \ar@{}[d]|(0.5){\overset{\tau}{\Longrightarrow}} & P'\ar[ld]^{\pi'} \\
 & \X. &
 }
 \]
\end{dfn}

For two arrows $(\phi,\tau):(\pi,\sigma) \to (\pi',\sigma')$ and $(\phi',\tau'):(\pi',\sigma') \to (\pi'',\sigma'')$ in $\PG(\X)$, the composition $(\phi',\tau')(\phi,\tau)$ is defined by $(\phi'\circ\phi,(\tau'\circ\id_\phi)\bullet\tau)$, where $\circ$ is the horizontal composition and $\bullet$ is the vertical composition. The identity arrow of $(\pi:P\to\X,\sigma)$ is $(\id_P,\id_\pi)$. Then $\PG(\X)$ form a groupoid.

\begin{remark}
 For a manifold $X$, $\P^G(X)$ is the category of principal $G$-bundles over $X$.
\end{remark}

\begin{remark}
 A $G$-stack is defined by Romagny \cite{romagny05:_group} as a generalisation of $G$-spaces. If we regard a stack $\X$ as a trivial $G$-stack, then any object $(\pi,\sigma)$ of $\P^G(\X)$ gives a morphisms of $G$-stack. Moreover for any arrow $(\phi,\tau):(\pi,\sigma) \to (\pi',\sigma')$ in $\P^G(\X)$, the 2-morphism of stacks $\tau$ gives a 2-morphism of $G$-stacks $\pi \to \pi'\circ\phi$.
\end{remark}

Let $(\pi^P:P \to \X,\sigma)$ be an object of $\PG(\X)$. For any representable morphism $F:\Y \to \X$, we can construct an object of $\PG(\Y)$ by taking a pullback. Now we have a manifold $Q$ and the 2-cartesian square
\begin{equation}
 \label{eqn:pullback-of-principal-G-bundle}
  \xymatrix@C=15pt@R=10pt{
  Q \ar[rrr]^{\epsilon}\ar[dd]_{\pi} & & & P \ar[dd]^{\pi^P} \\
   & & & \\
  \Y \ar[rrr]_{F} & & \ar@{=>}[ru]^{\alpha} & \X.
  }
\end{equation}
Define $\epsilon':Q \times G \to P$ by $\epsilon'(q,g)=\epsilon(q)g$. Then the map 
\[
 \Ob{Q \times G} \to \Ar{\X};\ (q,g) \mapsto \sigma(\epsilon(q),g)\alpha(q)
\]
gives a 2-morphism of stacks $F\circ(\pi\circ\pr_1) \to \pi^P\circ\epsilon'$. Since the above diagram is 2-cartesian, there is a unique smooth map $\psi':Q \times G \to Q$ and a 2-morphism of stacks $\sigma':\pi\circ\pr_1 \to \pi\circ\psi$ such that $\epsilon\circ\psi'=\epsilon'$ and 
\begin{equation}
 \label{eqn:property-for-sigma'}
 \alpha(\psi'(q,g))F(\sigma'(q,g))=\sigma(\epsilon(q),g)\alpha(q) 
\end{equation}
for any $(q,g) \in Q \times G$.
\begin{prop}
 \label{prop:pullback-equipped-with-G-action}
 The maps $\psi':Q \times G \to Q$ defines a right $G$-action on $Q$.
\end{prop}

The following lemma is a direct conclusion of the universality of the 2-cartesian square (\ref{eqn:pullback-of-principal-G-bundle}).
\begin{lem}
 \label{lem:universality}
 Let $\tilde\pi:R\to\Y$ and $\tilde\epsilon:R \to P$ be morphisms of stacks. 
 Suppose we have two pairs $(f_1,\beta_1)$ and $(f_2,\beta_2)$ of a smooth map $f_i:R \to Q$ with $\tilde\epsilon=\epsilon \circ f_i$ and a 2-morphism of stacks $\beta_i:\tilde\pi \to \pi\circ f_i$ $(i=1,2)$:
 \[
 \xymatrix@C=25pt@R=10pt{
 && R \ar[dd]^{f_i}\ar[lldd]_{\tilde\pi}\ar[rrdd]^{\tilde\epsilon} && \\
 & \ar@{=>}[d]^{\ \beta_i} &&& \\
 \Y && Q \ar[ll]^{\pi}\ar[rr]_{\epsilon} && P.
 }
 \]
 If $\alpha(f_1(r))F(\beta_1(r)) = \alpha(f_2(r))F(\beta_2(r))$ for any $r \in R$, then $f_1 = f_2$ and $\beta_1 = \beta_2$.
\end{lem}

\begin{proof}%
 [Proof of Proposition \ref{prop:pullback-equipped-with-G-action}]
 First we show that $\psi'(q,1) = q$ holds for any $q \in Q$. Define $s:Q \to Q \times G$ by $s(q)=(q,1)$. Let $(f_1,\beta_1) = (\psi \circ s,\sigma'\circ\id_s)$ and $(f_2,\beta_2) = (\id_Q,\id_\pi)$, then we obtain the following 2-commutative diagram ($i=1,2$):
 \[
 \xymatrix@C=25pt@R=10pt{
 && Q \ar[dd]^{f_i}\ar[lldd]_{\pi}\ar[rrdd]^{\epsilon} && \\
 & \ar@{=>}[d]^{\ \beta_i} &&& \\
 \Y && Q \ar[ll]^{\pi}\ar[rr]_{\epsilon} && P.
 }
 \]
Then the identity (\ref{eqn:property-for-sigma'}) implies $\alpha(f_1(q))F(\beta_1(q)) = \alpha(f_2(q))F(\beta_2(q))$ for any $q \in Q$. Thus Lemma \ref{lem:universality} implies $f_1=f_2$ i.e.\! $\psi'(q,1)=q$.

 Next we show that $\psi'(\psi'(q,g),h)=\psi'(q,gh)$ holds for any $(q,g,h) \in Q \times G \times G$. Define $s_1:Q \times G \times G \to Q \times G$ by $s_1(q,g,h)=(q,gh)$ and $s_2:Q \times G \times G \to Q \times G$ by $s_2(q,g,h)=(\psi'(q,g),h)$. Let $(f_1,\beta_1) = (\psi' \circ s_1, \sigma'\circ \id_{s_1})$ and $(f_2,\beta_2) = (\psi' \circ s_2, (\sigma'\circ\id_{s_2})\bullet(\sigma'\circ\pr_{12}))$, then we obtain the following 2-commutative diagram ($i=1,2$):
 \[
 \xymatrix@C=25pt@R=10pt{
 && Q \times G \times G
 \ar[dd]^{f_i}
 \ar@/_/[lldd]_{\pi\circ\pr_1}
 \ar@/^/[rrdd]^{\quad (q,g,h) \mapsto \epsilon(q)gh} && \\
 & \ar@{=>}[d]^{\ \beta_i} &&& \\
 \Y && Q \ar[ll]^{\pi}\ar[rr]_{\epsilon} && P.
 }
 \]
 Then the identities (\ref{eqn:2-morph-in-group-action}) and (\ref{eqn:property-for-sigma'}) imply $\alpha(f_1(q))F(\beta_1(q)) = \alpha(f_2(q))F(\beta_2(q))$ for any $q \in Q$. Thus Lemma \ref{lem:universality} implies $f_1=f_2$ i.e.\! $\psi'(q,gh)=\psi'(\psi'(q,g),h)$.
\end{proof}
From now on, $\psi'(q,g)$ is denoted by $qg$. Then we can write the identity (\ref{eqn:property-for-sigma'}) in the form 
\begin{equation}
 \label{eqn:modified-property-for-sigma'}
 \alpha(qg)F(\sigma'(q,g))=\sigma(\epsilon(q),g)\alpha(q) 
\end{equation}
for any $(q,g) \in Q \times G$. The identity $\epsilon\circ\psi'=\epsilon'$ implies that the smooth map $\epsilon:Q \to P$ is $G$-equivariant. Note that we also obtain the following identity 
\begin{equation}
 \label{eqn:sigma'-2-morph-in-group-action}
 \sigma'(q,gh)=\sigma'(qg,h)\sigma'(q,g)
\end{equation}
for any $(q,g,h) \in Q \times G \times G$ in the proof of the above proposition.

\begin{prop}
 \label{prop:G-bundle-is-stable-under-taking-pullback}
 The pair $(\pi:Q \to \Y,\sigma')$ is an object of $\PG(\Y)$.
\end{prop}
\begin{proof}
It is sufficient to see that the morphism of stacks
\[
\Phi: Q \times G \to Q \fp{\pi}{\pi} Q;\
\begin{cases}
 (q,g) \to (q,qg,\sigma'(q,g)) & \text{for objects,}\\
 a \to (a,a) & \text{for arrows}
\end{cases}
\]
is an equivalence (over each manifold $U$). First we show that the morphism $\phi$ is fully faithful over $U$. Let $(q_1,g_1), (q_2,g_2) \in Q \times G$ be objects over $U$. Suppose that we have an arrow $(q_1,q_1g_1,\sigma'(q_1,g_1)) \to (q_2,q_2g_2,\sigma'(q_2,g_2))$ in $(Q \fp{\pi}{\pi} Q)_U$. Since such an arrow is an identity arrow, we have $(q_1,q_1g_1,\sigma'(q_1,g_1))=(q_2,q_2g_2,\sigma'(q_2,g_2))$. Moreover this implies that $(\epsilon(q_1),\epsilon(q_1)g_1,\sigma(\epsilon(q_1),g_1)))=(\epsilon(q_2),\epsilon(q_2)g_2,\sigma(\epsilon(q_2),g_2)))$ in $(P\fp{\pi^P}{\pi^P}P)_U$. Therefore $g_1=g_2$ and the morphism $\Phi$ is fully faithful over $U$.

Next we show that the morphism $\Phi$ is essentially surjective over $U$. Let $(q_1,q_2,a)$ be an object of $(Q \fp{\pi}{\pi} Q)_U$. Since $(\epsilon(q_1),\epsilon(q_2),\alpha(q_2)F(a)\alpha(q_1)^{-1})$ is an object of $(P \fp{\pi^P}{\pi^P}P)_U$, there is $(p,g) \in (P \times G)_U$ such that $(\epsilon(q_1),\epsilon(q_2),\alpha(q_2)F(a)\alpha(q_1)^{-1})=(p,pg,\sigma(p,g))$. Now the pair $(\sigma'(q_1,g)a^{-1},1)$ gives an arrow from $(\pi(q_2),\epsilon(q_2),\alpha(q_2))$ to $(\pi(q_1g),\epsilon(q_1g),\alpha(q_1g))$. In fact, 
\begin{align*}
 \alpha(q_1g)F(\sigma'(q_1,g)a^{-1}) 
 &= \alpha(q_1g)F(\sigma'(q_1,g))F(a)^{-1} & \\
 &= \sigma(p,g)\alpha(q_1)F(a)^{-1} &(\because\ (\ref{eqn:modified-property-for-sigma'}).) \\
 &= \alpha(q_2). &(\because\ \sigma(p,g)=\alpha(q_2)F(a)\alpha(q_1)^{-1}.) 
\end{align*}
Since the diagram (\ref{eqn:pullback-of-principal-G-bundle}) is 2-cartesian, we can conclude that $q_2=q_1g$ and $a=\sigma'(q_1,g)$, i.e. $(q_1,q_2,a)=(q_1,q_1g,\sigma'(q_1,g))$. Therefore the morphism $\Phi$ is essentially surjective over $U$.
\end{proof}

\subsection{Automorphisms in $\PG(\X)$}
\label{sec:Autmorphisms}

Let $\X$ be a differentiable stack. For an object $(\pi:P \to \X,\sigma)$ of $\PG(\X)$, we consider the group of automorphisms $\Aut(\pi,\sigma)$ of $(\pi,\sigma)$ in $\PG(\X)$. Let $(\phi,\tau)\in \Aut(\pi,\sigma)$. Since $\tau$ is a 2-morphism of stacks $\pi\circ\id_P \to \pi\circ\phi$, we have a unique smooth map $\gamma_{(\pi,\sigma)} = \gamma: P \to G$ making the following diagram commutative.
\[
 \xymatrix@C=15pt@R=10pt{
 P
 \ar[rd]^{(\id_P,\gamma)} 
 \ar@/^1pc/[drrrr]^{\phi} 
 \ar@/_1pc/[dddr]_{\id_P}
 & & & & \\
 & P \times G \ar[rrr]_{\psi} \ar[dd]^{\pr_1}
 & & & P \ar[dd]^{\pi} \\
 & & & & \\
 & P \ar[rrr]_{\pi} & & \ar@{=>}[ru]^{\sigma} & \X.
 }
\]
Therefore $\phi(p)=p\gamma(p)$ for any $p \in P$. Since $\phi$ is $G$-equivariant, $p\gamma(p)g=pg\gamma(pg)$ for any $p \in P$ and $g \in G$. We also have $\tau(p)=\sigma(p,\gamma(p))$ for $p \in P$.

Now we assume that the stack $\X$ is type-R (cf.\! Definition \ref{dfn:type-R}). Since $[P/G]$ is equivalent to $\X$, the $G$-action of $P$ must be locally free and effective. Therefore $\gamma(pg)=g^{-1}\gamma(p)g$ holds for any $g \in G$.
Define $\G(P)$ by
\begin{equation}
 \label{eqn:gauge-group}
\G(P) = \{ \gamma \in \mathcal{C}^\infty(P,G) \mid \gamma(pg)=g^{-1}\gamma(p)g \ \text{for all}\ p \in P \ \text{and}\ g \in G\}. 
\end{equation}
The space $\G(P)$ inherits the group structure from $G$. It is easy to see that the above assignment $(\phi,\tau) \mapsto \gamma_{(\phi,\tau)}$ gives an injective group anti-homomorphism from $\Aut(\pi,\sigma)$ to $\G(P)$. Suppose $\gamma \in \G(P)$. Putting $\phi(p)=p\gamma(p)$ and $\tau(p)=\sigma(p,\gamma(p))$, we can easily check that $(\phi,\tau) \in \Aut(\pi,\sigma)$. Thus the group anti-homomorphism is surjective. 

\begin{prop}
 If $\X$ is type-R, then $\Aut(\pi:P\to \X,\sigma)$ is anti-isomorphic to the group $\G(P)$.
\end{prop}

\subsection{The category of representable morphisms}

Let $M$ be a right $G$-space and $\X$ a differentiable stack. Making use of the category $\P^G(\X)$ we construct a category which is equivalent to the category of representable morphisms from $\X$ to the quotient stack $[M/G]$.

\begin{dfn}
 We define a category $\RG(\X,M)$ as follows.
 \begin{itemize}
  \item \squashup
	An object of $\RG(\X,M)$ is a triple $(\pi,\sigma,\epsilon)$ such that $(\pi,\sigma)$ is an object of $\PG(\X)$ and $\epsilon$ is a $G$-equivariant map $P \to M$. 
  \item \squashup
	An arrow $(\pi,\sigma,\epsilon) \to (\pi',\sigma',\epsilon')$ in $\RG(\X,M)$ is an arrow $(\phi,\tau):(\pi,\sigma) \to (\pi',\sigma')$ in $\PG(\X)$ satisfying $\epsilon'\circ\phi = \epsilon$.
  \item \squashup 
	The composition of two arrows $(\phi,\tau):(\pi,\sigma,\epsilon) \to (\pi',\sigma',\epsilon')$ and $(\phi',\tau'):(\pi',\sigma') \to (\pi'',\sigma'')$ is defined by the composition $(\phi',\tau')(\phi,\tau)$ in $\PG(\X)$.
  \item \squashup 
	The identity arrow of $(\pi,\sigma,\epsilon)$ is defined by the identity arrow $1_{(\pi,\sigma)}$ in $\PG(\X)$.
 \end{itemize}
\end{dfn}

The category of representable morphisms from $\X$ to $[M/G]$ is denoted by $\MorRep(\X,[M/G])$. 

\begin{thm}
 \label{thm:category-of-representable-morphisms}
 The category $\MorRep(\X,[M/G])$ is equivalent to the category $\RG(\X,M)$.
\end{thm}

Define a functor $\Psi:\MorRep(\X,[M/G])\to\RG(\X,M)$ as follows. For each representable morphism $F:\X\to [M/G]$, we can choose a manifold $P^F$, a morphism $\pi^F:P^F \to \X$ and a smooth map $\epsilon^F:P^F \to M$ so that the following square is 2-cartesian:
\begin{equation*}
  \xymatrix@C=15pt@R=10pt{
  P^F \ar[rrr]^{\epsilon^F}\ar[dd]_{\pi^F} & & & M \ar[dd]^{\pi^M} \\
   & & & \\
  \X \ar[rrr]_{F} & & \ar@{=>}[ru]^{\alpha^F} & [M/G].
  }
\end{equation*}
Propositions \ref{prop:pullback-equipped-with-G-action} and \ref{prop:G-bundle-is-stable-under-taking-pullback} say that there is a 2-morphism of stacks $\sigma^F: \pr_1\circ\pi^F \to \psi\circ\pi^F$ such that $(\pi^F:P^F \to \X,\sigma^F)$ is an object of $\PG(\X)$ and $\epsilon^F$ is $G$-equivariant i.e.\! $(\pi^F,\sigma^F,\epsilon^F)$ is an object of $\RG(\X,M)$. We define $\Psi(F)$ by $(\pi^F,\sigma^F,\epsilon^F)$ (cf.\! Remark \ref{remark:category-PG-and-quotient-stack}). Note that the 2-morphism of stacks $\sigma^F$ satisfies
\begin{equation}
 \label{eqn:sigmaF-as-pullback}
  \alpha^F(pg)F(\sigma^F(p,g))=\sigma^M(\epsilon^F(p),g)\alpha^F(p)
\end{equation}
for any $(p,g) \in P \times G$.

Let $F_1$ and $F_2$ be representable morphisms $\X \to [M/G]$ and $\theta:F_1 \to F_2$ a 2-morphisms of stacks. Since we have a 2-morphism of stacks $\alpha^{F_1} \bullet (\theta^{-1}\circ\id_{\pi^{F_1}}):F_2\circ\pi^{F_1}\to\epsilon^{F_1}\circ\pi^M$, we obtain a unique smooth map $\phi^\theta:P^{F_1} \to P^{F_2}$ satisfying $\epsilon^{F_2}\circ\phi=\epsilon^{F_1}$ and a unique 2-morphism of stacks $\tau^\theta:\pi^{F_1} \to \pi^{F_2}\circ\phi$:
\[
 \xymatrix@C=15pt@R=10pt{
 P^{F_1} 
 \ar[rd]^{\phi^\theta} 
 \ar@/^1pc/[drrrr]^{\epsilon^{F_1}} 
 \ar@/_1pc/[dddr]_{\pi^{F_1}}
 & & & & \\
 & P^{F_2} \ar[rrr]^{\epsilon^{F_2}} \ar[dd]^{\pi^{F_2}}
 & & & M \ar[dd]^{\pi^M} \\
 & \ar@{}[l]|{\quad\ \overset{\tau^\theta}{\Longrightarrow}} & & & \\
 & \X \ar[rrr]_{F_2} & & \ar@{=>}[ru]^{\alpha^{F_2}} & [M/G].
 }
\]
Here the 2-morphism of stacks $\tau^\theta$ satisfies
\begin{equation}
 \label{eqn:property-of-tau}
 \alpha^{F_2}(\phi^\theta(p))F_2(\tau^\theta(p))=\alpha^{F_1}(p)\theta(\pi^{F_1}(p))^{-1} 
\end{equation}
for any $p \in P^{F_1}$. 

We see that $(\phi^\theta,\tau^\theta)$ is an arrow $(\pi^{F_1},\sigma^{F_1},\epsilon^{F_1}) \to (\pi^{F_2},\sigma^{F_2},\epsilon^{F_2})$ in the category $\RG(\X,M)$. 
Let $(f_1,\beta_1)=(\phi^\theta\circ\psi,(\tau^\theta\circ\id_\psi)\bullet\sigma^{F_1})$ and $(f_2,\beta_2)=(\psi\circ(\phi^\theta\times\id_G),(\sigma^{F_2}\circ\id_{\phi^\theta\times\id_G})\bullet(\tau^\theta\circ\id_{\pr_1}))$. Then we have the following 2-commutative diagrams ($i=1,2$):
\[
\xymatrix@C=25pt@R=10pt{
&& P^{F_1} \times G
\ar[dd]^{f_i}
\ar@/_/[lldd]_{\pi^{F_1}\circ\pr_1}
\ar@/^/[rrdd]^{\quad (p,g) \mapsto \epsilon^{F_1}(p)g} && \\
& \ar@{=>}[d]^{\ \beta_i} &&& \\
\X && P^{F_2} \ar[ll]^{\pi^{F_2}}\ar[rr]_{\epsilon^{F_2}} && P.
}
\]
Then for $(p,g) \in P^{F_1} \times G$ we have 
\begin{align*}
 \alpha^{F_2}(f_1(p,g))F_2(\beta_1(p,g))
 &= \alpha^{F_2}(\phi^\theta(pg))F_2(\tau^\theta(pg))F_2(\sigma^{F_1}(p,g)) \\
 &= \alpha^{F_1}(pg)\theta(\pi^{F_1}(pg))^{-1}F_2(\sigma^{F_1}(p,g)) 
 &(\because (\ref{eqn:property-of-tau}).) \\
 &= \alpha^{F_1}(pg)F_1(\sigma^{F_1}(p,g))\theta(\pi^{F_1}(p))^{-1} \\
 &= \sigma^M(\epsilon^{F_1}(p),g)\alpha^{F_1}(p)\theta(\pi^{F_1}(p))^{-1}.
 &(\because (\ref{eqn:sigmaF-as-pullback}).)
\end{align*}
On the other hand, we have 
\begin{align*}
 \alpha^{F_2}(f_2(p,g))F_2(\beta_2(p,g))
 &= \alpha^{F_2}(\phi^\theta(p)g)F_2(\sigma^{F_2}(\phi^\theta(p),g))F_2(\tau^\theta(p)) \\
 &= \sigma^M(\epsilon^{F_2}(\phi^\theta(p)),g)\alpha^{F_2}(\phi^\theta(p))F_2(\tau^\theta(p)) 
 &(\because (\ref{eqn:sigmaF-as-pullback}).) \\
 &= \sigma^M(\epsilon^{F_1}(p),g)\alpha^{F_1}(p)\theta(\pi^{F_1}(p))^{-1}. 
 &(\because (\ref{eqn:property-of-tau}).)
\end{align*}
Applying Lemma \ref{lem:universality}, we conclude that $f_1=f_2$ and $\beta_1=\beta_2$, i.e.\! for $(p,g) \in P^{F_1} \times G$ we have $\phi^\theta(pg)=\phi^\theta(p)g$ and $\tau^\theta(pg)\sigma^{F_1}(p,g)=\sigma^{F_2}(\phi^\theta(p),g)\tau^\theta(p)$. Similarly we can see that $\phi^\theta$ is a diffeomorphism. (In fact, $\phi^{(\theta^{-1})}$ is the inverse of $\phi^\theta$.)
Therefore $(\phi^\theta,\tau^\theta)$ is an arrow $(\pi^{F_1},\sigma^{F_1},\epsilon^{F_1}) \to (\pi^{F_2},\sigma^{F_2},\epsilon^{F_2})$ in the category $\RG(\X,M)$. We define $\Psi(\theta)=(\phi^\theta,\tau^\theta)$.

\begin{lem}
 The functor $\Psi:\MorRep(\X,[M/G])\to\RG(\X,M)$ is faithful.
\end{lem}
\begin{proof}
 Let $\theta$ and $\theta'$ be two 2-morphisms of stacks from $F_1$ to $F_2$ in $\MorRep(\X,[M/G])$. Suppose that we have an arrow $(\phi,\tau):\Psi(\theta)\to\Psi(\theta')$ in $\RG(\X,M)$. Because of Identity (\ref{eqn:property-of-tau}), we have 
\[
 \theta(\pi^{F_1}(p)) 
 = F_2(\tau(p))^{-1}\alpha^{F_2}(\phi(p))^{-1}\alpha^{F_1}(p)
 = \theta'(\pi^{F_1}(p))
\]
for any $p \in P$. Therefore it suffices to show that for any object $x \in \X$ the arrow $\theta(x)$ can be calculated without $\theta$. Let $x \in \X_U$ and $\{U_i \to U\}$ be a covering family of $U$ in $\Diff$. Since $[M/G]$ is a stack, $\theta(x)$ can be uniquely determined by its pullbacks $f_i^*\theta(x)$. Therefore it suffices to show that $f_i^*\theta(x)$ can be calculated without $\theta$.

Choosing a manifold $P_x$ equivalent to $U \fp{x}{\pi^{F_1}} P^{F_1}$, we have a principal $G$-bundle $\pi_x: P_x \to U$, a $G$-equivariant map $\epsilon_x: P_x \to P_{F_1}$ and a 2-morphism of stacks $\beta_x:x\circ\pi_x \to \pi^{F_1}\circ\epsilon_x$. We may assume that the pullback of the principal $G$-bundle $\pi_x: P_x \to U$ via $f_i$ is trivial:
\[
\xymatrix@C=10pt@R=10pt{
 U_i \times G \ar[rrr]^{}\ar[dd]_{\pr_1} & & & 
 P_x \ar[rrr]^{\epsilon_x}\ar[dd]_{\pi_x} & & & P^{F_1} \ar[dd]^{\pi^{F_1}} \\
 & & & 
 & & & \\
 U_i \ar[rrr]_{f_i} \ar[rrruu]_{s_i} & & & 
 U \ar[rrr]_{x} & & \ar@{=>}[ru]^{\beta_x} & \X.
 }
\]
Here $s_i:U_i \to P_x$ is a smooth map corresponding to a trivialisation. Using $s_i \in (P_x)_{U_i}$, we obtain the following arrows in $[M/G]_{U_i}$:
\[
 \xymatrix@C=50pt@R=20pt{
 f_i^*F_1(x) \ar[d]^{f_i^*\theta(x)} &
 F_1(f_i^*x) \ar[d]^{\theta(f_i^*x)} \ar[l] \ar[r]^{F_1(\beta_x(s_i))} &
 F_1(\pi^{F_1}(p_i)) \ar[d]^{\theta(\pi^{F_1}(p_i))}  \\
 f_i^*F_2(x) & F_2(f_i^*x) \ar[l] \ar[r]_{F_2(\beta_x(s_i))} & F_2(\pi^{F_1}(p_i)). \\
 }
\]
Because of Identity (\ref{eqn:property-of-tau}), $\theta(\pi^{F_1}(p_i))$ can be calculated without $\theta$, and so is $f_i^*\theta(x)$.
\end{proof}

Let $F$ be a representable morphism from $\X$ to $[M/G]$. Making use of the atlas $P^F \rtimes G$ of $\X$, we can describe $F$ in an explicit way:
\begin{prop}
 \label{prop:explicit-formula-for-representable-morphism}
 Let $F \in \MorRep(\X,[M/G])$. If we identify $\X$ with $[P^F/G]$, then the morphism of stacks $F$ is isomorphic to the morphism of stacks $\hat{F}$ defined by 
 \begin{equation*}
  \hat{F}: [P^F/G] \to [M/G];\ 
   \begin{cases}
    (\pi,\epsilon) \mapsto (\pi,\epsilon^F\circ\epsilon) & \text{for objects,}\\
    (f,\tilde{f}) \mapsto (f,\tilde{f}) & \text{for arrows}.
   \end{cases}
 \end{equation*}
\end{prop}

\begin{remark}
 There is an embedding $\B$ of the bicategory $\Bi$ of Lie groupoids and bibundles into the 2-category $\StDiff$ of stacks over $\Diff$. For a Lie groupoid $\gpoid{X}$, $\B(\gpoid{X})$ is the category of principal ($\gpoid{X}$)-bundles.  For a bibundle $f_P$ from $\gpoid{X}$ to $\gpoid{Y}$, the morphism of stacks $\B(f_P):\B(\gpoid{X})\to\B(\gpoid{Y})$ is defined as follows. Any principal ($\gpoid{X}$)-bundle $\alpha:Q \to M$ naturally defines a bibundle $f_Q$ from ${M}\!\rightrightarrows\!{M}$ to $\gpoid{X}$. The composition $f_P \circ f_Q$ includes a principal ($\gpoid{Y}$)-bundle, and it is the image of $f_Q$ via $\B(f_P)$. The image of the embedding $\B$ is the $2$-category of differentiable stacks (without the condition to be proper). For details of the embedding $\B$, see Lerman \cite[\S 4]{lerman10:_orbif}.
\end{remark}

\begin{proof}
 A bibundle corresponding a morphism of stacks can be obtained by by taking the following pullback \cite{lerman10:_orbif}:
\[
\xymatrix@C=10pt@R=10pt{
 P^F \times G \ar[rrr]^{\psi}\ar[dd]_{\pr_1} & & & 
 P^F \ar[rrr]^{\epsilon^F}\ar[dd]^{\pi^F} & & & M \ar[dd]^{\pi^M} \\
 & & &
 & & & \\
 P^F \ar[rrr]_{\pi^F} & & \ar@{=>}[ru]^{\sigma^F} & 
 \X \ar[rrr]_{F} & &  \ar@{=>}[ru]^{\alpha^F} & [M/G].
 }
\]
The triple $(\pr_1,P^F,\epsilon^F\circ\psi)$ is a bibundle from $P^F \rtimes G$ to $M \rtimes G$ corresponding to $F:\X \to [M/G]$. It is easy to see that the morphism of stacks $\B(\pr_1,P^F,\epsilon^F\circ\psi):[P^F/G] \to [M/G]$ is isomorphic to $\hat{F}$.
\end{proof}

\begin{lem}
 The functor $\Psi:\MorRep(\X,[M/G])\to\RG(\X,M)$ is essentially surjective.
\end{lem}

\begin{proof}
 Let $(\pi:P \to \X,\sigma,\epsilon)$ be an object of $\RG(\X,M)$ and identify $\X$ with $[P/G]$. Consider the morphism of stacks defined by 
 \[
 F: [P/G] \to [M/G];\ 
 \begin{cases}
  (\pi',\epsilon') \mapsto (\pi',\epsilon\circ\epsilon') & \text{for objects,}\\
 (f,\tilde{f}) \mapsto (f,\tilde{f}) & \text{for arrows}.
 \end{cases}
 \]
It is easy to see that the following square is 2-cartesian: 
\[
 \xymatrix{
 P \ar[d]_{\pi^P}\ar[r]^{\epsilon} & M \ar[d]^{\pi^M}\\
 [P/G] \ar[r]_{F} & [M/G].
 }
\]
(The identity transformation gives a 2-morphism of stack in the above square.) This also implies that $F$ is representable. Then we have a unique smooth map $\phi:P^F \to P$ and a 2-morphism of stacks $\tau:\pi^P \to \phi\circ\pi$ such that the following diagram is 2-commutative.
\[
 \xymatrix@C=15pt@R=10pt{
 P^{F} 
 \ar[rd]^{\phi} 
 \ar@/^1pc/[drrrr]^{\epsilon^F} 
 \ar@/_1pc/[dddr]_{\pi^F}
 & & & & \\
 & P \ar[rrr]^{\epsilon} \ar[dd]^{\pi}
 & & & M \ar[dd]^{\pi^P} \\
 & \ar@{}[l]|{\quad\ \overset{\tau}{\Longrightarrow}} & & & \\
 & [P/G] \ar[rrr]_{F} & & & [M/G].
 }
\]
Modifying the above discussion for $\Psi(\theta)$ slightly, we can show that the pair $(\phi,\tau)$ gives an arrow from $\Psi(F)$ to $(\pi,\sigma,\epsilon)$ in $\RG(\X,[M/G])$.
\end{proof}

\begin{lem}
 The functor $\Psi:\MorRep(\X,[M/G])\to\RG(\X,M)$ is full.
\end{lem}
\begin{proof}
 Let $F_1, F_2 \in \MorRep(\X,[M/G])$. Suppose we have an arrow $(\phi,\tau)$ from $\Psi(F_1)$ to $\Psi(F_2)$. Since $\phi$ is a $G$-equivariant diffeomorphism, it is obvious that the following morphism $E$ is an equivalence:
 \[
 E: [P^{F_1}/G] \to [P^{F_2}/G];\ 
 \begin{cases}
  (\pi,\epsilon) \mapsto (\pi,\phi\circ\epsilon) & \text{for objects,}\\
  (f,\tilde{f}) \mapsto (f,\tilde{f}) & \text{for arrows}.
 \end{cases}
 \]
 We identify $[P^{F_2}/G]$ with $[P^{F_1}/G]$ via the above morphism. Then we can also identify $F_1$ and $F_2$ because of Proposition \ref{prop:explicit-formula-for-representable-morphism} and $\epsilon^{F_2}\circ\phi=\epsilon^{F_1}$. Therefore the functor $\Psi$ is full.
\end{proof}

\section{Symplectic vortex equation}

In this section, we define symplectic vortex equations over a compact Riemann surface $\mathbf\Sigma$, regarding $\mathbf\Sigma$ as a type-R stack $\X$. First we fix notations for an orbifold Riemann surface. Secondly we see the definition of (representable) pseudo-holomorphic curve from $\X$ to a Marsden--Weinstein quotient (also known as a symplectic quotient). After that we define symplectic vortex equations and an energy functional. Finally we discuss the moduli space of solutions to the equations for a linear action of the circle group $S^1$ on the complex plane $\C$.

In this section, $G$ is a compact and connected Lie group. The Lie algebra of $G$ is denoted by $\g$. We fix an invariant inner product $\pair{\ ,\ }$ on $\g$ to identify $\g$ with its dual $\g^\dual$.

A Hamiltonian $G$-space is a symplectic manifold $(M,\omega)$ equipped with a (right) $G$-action and a $G$-equivariant map $\mu:M \to \g$ satisfying 
\begin{equation}
 \mu(mg) = \Ad(g^{-1})\mu(m) \ \ \text{for any}\  m \in M\ \text{and}\ g \in G
\end{equation}
and 
\begin{equation}
 \label{eqn:moment-map}
 d\pair{\mu,\xi} = -\iota(\xi^M)\omega \ \ \text{for any}\ \xi \in \g.
\end{equation}
Here $\Ad$ is the adjoint representation of $G$. Note that the $G$-action preserves $\omega$ since $G$ is connected. This map $\mu$ is called a moment map for the $G$-action. 

\subsection{An orbifold Riemann surface, revisited}
\label{section:orbifold-Riem-surface-revisited}

Let $\mathbf\Sigma = (\Sigma,j,\mathbf{z},\mathbf{m})$ be an orbifold Riemann surface and $\X$ the type-R stack corresponding to $\mathbf\Sigma$ (\S \ref{sec:orbifold-Riemann-surface}) and $(\pi:P\to\X,\sigma)$ an object of $\PG(\X)$. Since $\X$ is type-R, the right $G$-action on $P$ is locally free and effective (cf. \S \ref{sec:Autmorphisms}). In particular the space of connections
\[
 \A(P) = \{
 A \in \Omega^1(P,\g) \mid 
 \iota(\xi^P)A = \xi\ (\forall \xi \in \g)
 \ \ \text{and}\ \ 
 \psi_g^*A=\Ad(g^{-1})A\ (\forall g \in G)
 \}
\]
is nonempty.

The sheaf over $P$ induced by the tangent sheaf $\Ts_\X$ is the sheaf of sections of the bundle $TP/\g^P$. Note that the bundle $TP/\g^P$ has a complex structure coming from the complex structure $j$ on $\X$, and we also denote by $j$ the complex structure on $TP/\g^P$. Since $j$ is compatible with the source and target maps, the identity $(j[v])g = j[vg]$ holds for any $v \in TP$ and $g \in G$. 

We will need local slices for integrations over $\X$ and \S \ref{section:circle-actions}. Let $p \in P$ and $(U,\phi,\Gamma)$ a triple of
\begin{enumerate}
 \item \squashup
       a finite subgroup $\Gamma \subset G$, which acts on $\C$ in an orthogonal linear way, 
 \item \squashup
       a $\Gamma$-invariant open neighbourhood $U$ of $0$ in $\C$, and
 \item \squashup
       a smooth map $\phi:U \to P$ satisfying $\phi(0)=p$ and $\phi(gu)=\phi(u)g^{-1}$.
\end{enumerate}
\squashup
The triple is called a (complex) slice at $p \in P$ if the map 
\[
 U \times_{\Gamma} G \to P;\ [u,a] \mapsto \phi(u)a
\]
is a $G$-equivariant diffeomorphism onto a $G$-invariant open neighbourhood of $p$. Here $\Gamma$ acts on $U \times G$ by $(u,a) \cdot g = (g^{-1}u,g^{-1}a)$ ($u \in U$, $a \in G$ and $g \in \Gamma$). 

We may choose a family $\{(U_i,\phi_i,\Gamma_i)\}_i$ of slices so that $\{\phi_i(U_i)G\}$ is an open cover of $P$. Let $Z_0$ be the disjoint union $\bigsqcup_i U_i$ and $\iota: Z_0 \to P$ is the map whose restriction to $U_i$ is the inclusion map $\phi_i:U_i \to P$. Let $Z_1$ be a fibred product of the diagram:
\[
 \xymatrix@C=60pt@R=20pt{
 Z_1 \ar[d]_{(\tgt,\src)}\ar[r]^{\iota_1} & P \times S^1 \ar[r]^{\pr_2} \ar[d]^{(\tgt,\src)} & S^1\\
 Z_0 \times Z_0 \ar[r]_{\iota \times \iota} & P \times P.&
 }
\]
Then we obtain a proper \'etale groupoid $\gpoid{Z}$ which is equivalent to $P \rtimes G$ in the bicategory $\Bi$ (cf.\! Lerman \cite[Definition 2.25]{lerman10:_orbif}). It is easy to see that the composition $\pi \circ \iota:Z_0 \to \X$ is an atlas of $\X$, and the groupoid $\gpoid{Z}$ is a presentation associated to the atlas. Note that the groupoid $\gpoid{Z}$ has a partition of unity \cite{behrend04:_cohom}. 

The space of 2-forms on $\X$ is given by the space of basic 2-forms on $P$
\[
 \Omega^2([P/G]) = 
\{ \eta \in \Omega^2(P) \mid 
 \iota(\xi^P)\eta = 0\ (\forall \xi \in \g)
 \ \text{ and }\ 
 \psi_g^* \eta = \eta\ (\forall g \in G) \}.
\]
In terms of $\gpoid{Z}$, the space of 2-forms on $\X$ is the space 
\[
 \Omega^2(\gpoid{Z}) = \{
 \theta \in \Omega^2(Z_0) \mid \src^*\theta = \tgt^*\theta
 \}
\]
and an explicit correspondence between $\Omega^2([P/G])$ and $\Omega^2(\gpoid{Z})$ is given by $\iota^*\eta=\theta$ for $\eta \in \Omega^2([P/G])$ and $\theta \in \Omega^2(\gpoid{Z})$.

We denote by $\rho:Z_1 \to S^1$ the composition of $\pr_2 \circ \iota_1$. The map $\rho$ satisfies $\iota(\tgt\alpha)\rho(\alpha)=\iota(\src\alpha)$ for any $\alpha \in Z_1$. The map
\[
 \epsilon:Z_0 \times S^1 \to P;\ (z,t) \mapsto \iota(z)t
\]
is an atlas of $P$, and a presentation associated to the atlas $\epsilon:Z_0 \times S^1 \to P$ is given by
\[
 Z_1 \times S^1 \rightrightarrows Z_0 \times S^1;\ \
 \src(\alpha,t)=(\src\alpha,t),\ \
 \tgt(\alpha,t)=(\tgt\alpha,\rho(\alpha)t).
\]
Note that the map $\epsilon:Z_0 \times S^1 \to P$ is $S^1$-equivariant with respect to the right $S^1$-action on $Z_0 \times S^1$ defined by $(z,t) \cdot t' = (z,tt')$.

\subsection{Pseudo-holomorphic curves in a Marsden--Weinstein quotient}
\label{section:pseudo-hol-map-in-symp-quot}

For details of (smooth) Marsden--Weinstein quotients, see McDuff--Salamon \cite{mcduff98:_introd}. Cieliebak--Gaio--Salamon \cite{cieliebak00:_j_hamil} is also helpful for our setting.

Let $(M,\omega)$ be a Hamiltonian $G$-space with a moment map $\mu:M \to \g$. Assume that $0 \in \g$ is a regular value of $\mu$. Then $\zl$ is a $G$-invariant submanifold of $M$ and the $G$-action on $\zl$ is locally free. Therefore the quotient stack $\MW{M}G=[\zl/G]$ is Deligne--Mumford. 

The quotient stack $\MW{M}G$ inherits a symplectic form from $(M,\omega)$ as follows. Let $\omega_\mu$ be the restriction of $\omega$ to $\zl$. A direct calculation shows that the 2-form $\omega_\mu$ is basic: $\omega_\mu \in \Omega^2(\MW{M}G)$.
 Since the kernel of the 2-form $\omega_\mu$ at $m \in \zl$ coincides with $\g^{\mu^{-1}(0)}(m)$, $\omega_\mu$ is nondegenerate and therefore $\omega_\mu$ gives a symplectic form on $\MW{M}G$. The symplectic Deligne--Mumford stack $\MW{M}G$ is called the Marsden--Weinstein quotient. 

Choose an almost complex structure $J$ compatible with the $G$-action and $\omega$ i.e.\! $J$ is an almost complex structure on $M$ such that $\psi_g \circ J = J \circ \psi_g$ for any $g \in G$ and the bilinear form 
\begin{equation}
 \label{eqn:metric-coming-from-omega-and-J}
 g(v_1,v_2) = \omega(v_1,Jv_2) 
\end{equation}
defines a Riemannian metric on $M$. Since the $G$-action on $\zl$ is locally free, $\g^{\mu^{-1}(0)}$ is a trivial subbundle of $T\zl$. The sheaf induced by the tangent sheaf $\Ts_{\MW{M}G}$ is the sheaf of sections of the quotient bundle $T\zl/\g^{\mu^{-1}(0)}$. The bundle $T\zl/\g^{\mu^{-1}(0)}$ inherits a complex structure from $TM$ and it gives rise to an almost complex structure on $\Ts_{\MW{M}G}$.


We consider a representable morphism from a compact orbifold Riemann surface $(\Sigma,j,\mathbf{z},\mathbf{m})$ to the Marsden--Weinstein quotient $\MW{M}G$. We regard the compact orbifold Riemann surface as a type-R stack $\X$ as Section \ref{sec:orbifold-Riemann-surface}. 

By Theorem \ref{thm:category-of-representable-morphisms}, a representable morphism $F$ from $\X$ to $\MW{M}G$ corresponds to an object of the category $\RG(\X,\zl)$, i.e.\! a pair of an object $(\pi:P\to\X,\sigma)$ of $\PG(\X)$ and a $G$-equivariant map $u:P \to \zl$:
\[
 \xymatrix@C=15pt@R=10pt{
 P \ar[rrr]^{u}\ar[dd]_{\pi} & & & \zl \ar[dd]^{\pi^M} \\
 & & & \\
 \X \ar[rrr]_{F} & & \ar@{=>}[ru]^{\alpha} & \MW{M}G.
 }
\]
We will stick with the presentation $P \rtimes G$ to describe geometric concepts on $\X$ (cf.\! \S \ref{section:orbifold-Riem-surface-revisited}). 

\begin{dfn}
 The representable morphism $F:\X\to\MW{M}G$ is said to be \emph{pseudo-holomorphic} if the map $du:TP/\g^P\to T\zl/\g^M$ is compatible with the complex structures: 
 \begin{equation}
  \label{eqn:pseudo-holomorphic-map}
   J \circ du = du \circ j \ \ \text{as a map}\ \ TP/\g^P\to T\zl/\g^{\zl}.
 \end{equation}
\end{dfn}
\begin{remark}
 This definition is independent of the choice of the object $(\pi,\sigma,u)$ of $\RG(\X,\mu^{-1}(0))$ up to isomorphism.
\end{remark}

The above condition (\ref{eqn:pseudo-holomorphic-map}) can be described more explicitly by using a connection. For a connection $A \in \A(P)$ the covariant derivative $\dA u \in \Omega^1(P,u^*TM)$ of $u$ is defined by
\[
 \dA u: TP \to u^*TM;\ 
 \dA u(v) = du(v) - A(v)^M(u(p)) \ \text{ for }\ v \in T_pP.
\]
We denote by $\brdA u$ the complex antiholomorphic part of $\dA u$:
\[
 \brdA u = \dfrac{1}{2}\left( \dA u + J \circ \dA u \circ j \right) \in \Omega^1(P,u^*TM). 
\]
Here $\dA u \circ j:TP \to TM$ is defined as follows. For $v \in T_pP$ there is a tangent vector $\tilde{v} \in T_pP$ such that $[\tilde{v}] = j[v]$. Define $(\dA u \circ j)(v) = \dA u(\tilde{v})$. This definition is independent of the choice of $\tilde{v}$ because the identity $\dA u (\xi^P(p)) = 0$ holds for any $\xi \in \g$. 

\begin{prop}
 \label{prop:Cauchy-Riemann-equation}
 A representable morphism $F:\X\to\MW{M}G$ is pseudo-holomorphic if and only if there is a connection $A \in \A(P)$ satisfying that $\brdA u=0$.
\end{prop}

\begin{proof}
First we note that we have the orthogonal decomposition 
\begin{equation}
 \label{eqn:orthogonal-decomposition}
 T_{u(p)}M =  N_{u(p)} \oplus \g^M(u(p)) \oplus J\g^M(u(p)).
\end{equation}
Here $N_{u(p)}$ is the orthogonal complement of $\g^M(u(p))$ in $T_{u(p)}\zl$. We can canonically identify the bundle $T\zl/\g^{\zl}$ with $N \to \zl$. We denote by $\nu_1$ and $\nu_2$ the projections from $N_{u(p)} \oplus \g^M(u(p))$ to the first and second component respectively. The condition (\ref{eqn:pseudo-holomorphic-map}) is equivalent to the identity $J\nu_1(du(p)v) = \nu_1(du(p)jv)$ for every $p \in P$ and $v \in T_pP$.

 Suppose a representable morphism $F:\X\to\MW{M}G$ to be pseudo-holomorphic. Define a 1-form $A \in \Omega^1(P,\g)$ by the composition $T_pP \to T_{u(p)}\zl \to \g^{\zl}(u(p)) \cong \g$. It is easy to see that $A \in \A(P)$ and $\brdA u=0$.

 Conversely suppose that we have a connection $A \in \A(P)$ satisfying $\brdA u=0$ i.e.
\begin{equation*}
 J(\dA u)v = (\dA u)jv
\end{equation*}
for any $v \in TP$. Applying the orthogonal decomposition (\ref{eqn:orthogonal-decomposition}) to the above identity, we can easily see that the identity $J\nu_1(du(p)v) = \nu_1(du(p)jv)$ holds for every $p \in P$ and $v \in T_pP$.
\end{proof}

\subsection{Symplectic vortex equations}
\label{section:SVE}

The symplectic vortex equations are defined for the following data:
\begin{itemize}
 \item A Hamiltonian $G$-space $(M,\omega)$ with a moment map $\mu:M \to \g$.
 \item \squashup
       An almost complex structure $J$ of $M$ compatible with the $G$-action.
 \item \squashup
       A compact orbifold Riemann surface $(\Sigma,j,\mathbf{z},\mathbf{m})$. We regard it as a stack $\X$.
 \item \squashup
       A volume form $\dvol_\X$ of $\X$.
 \item \squashup
       An object $(\pi^P:P\to\X,\sigma)$ of the category $\PG(\X)$. 
\end{itemize}

A volume form $\dvol_\X$ of $\X$ is a nowhere-vanishing 2-form on $\X$. Using a $G$-invariant Riemannian metric on $P$, we can easily see that a volume form of $\X$ always exists. Moreover for any 2-form $\alpha$ on $\X$ there is a unique function $f \in \O(\X)$ such that $\alpha = f\dvol_\X$. 

Since $\X$ is 2-dimensional, a choice of volume form $\dvol_\X$ gives an orientation on $\X$. On the other hand, the complex structure $j$ on $\X$ also gives an orientation on $\X$. Thus we choose a volume form $\dvol_\X$ so that both orientations are the same. More explicitly, we assume that the following inequality holds:
\begin{equation}
 \dvol_\X(v,\tilde{v}) > 0 \ \text{ for any nonzero }\ [v] \in T_pP/\g^P.
\end{equation}
Here $\tilde{v}$ is a tangent vector at $p$ satisfying $[\tilde{v}]=j[v]$, and $[\ ]$ is the projection map from $TP$ to the quotient bundle $TP/\g^P$. The left hand side of the inequality is independent of the choice of $\tilde{v}$ since $\dvol_\X$ is basic.


Fix a volume form $\dvol_\X$ of $\X$. We define the Hodge $*$-operator by
\[
 *: \Omega^2(\X) \to \Omega^0(\X);\ *\alpha = *(f\dvol_\X) = f.
\] 
The Hodge $*$-operator can naturally extend to a map $\Omega^2(\X,\g) \to \Omega^0(\X,\g)$. 

\begin{dfn}
 \label{dfn:symplectic-vortex-equations}
 The \emph{symplectic vortex equations} are the partial differential equations
 \begin{equation}
  \label{eqn:symplectic-vortex-equations}
   \brdA u = 0, \quad *F_A + \mu(u) = 0
 \end{equation}
 for $(A,u) \in \A(P) \times \Cinf_G(P,M)$. Here $\Cinf_G(P,M)$ is the space of $G$-equivariant smooth maps from $P$ to $M$. Note that the curvature $F_A$ belongs to $\Omega^2(\X,\g)$.
\end{dfn}

\begin{remark}
 \label{remark:adiabatic-limit}
 Assume that $0 \in \g$ is a regular value of the moment map $\mu$. If we replace $\dvol_\X$ with $\epsilon^{-2}\dvol_\X$ for $\epsilon > 0$, then the equations (\ref{eqn:symplectic-vortex-equations}) become
 \[
 \brdA u = 0, \quad *F_A + \epsilon^{-2}\mu(u) = 0.
 \]
 The limit of the above equations of $\epsilon$, as $\epsilon$ approaches to $0$, are nothing but the equations for pseudo-holomorphic morphisms from $\X$ to $\MW{M}G$ (cf. Proposition \ref{prop:Cauchy-Riemann-equation}):
 \[
 \brdA u = 0, \quad \mu(u) = 0.
 \]
\end{remark}

We define the the \emph{energy functional} (or the action functional) $E:\A(P)\times\Cinf_G(P,M)\to\R$ by 
\begin{equation}
 E(A,u)= \frac{1}{2}\int_{\X} \left( |\dA u|^2 + |F_A|^2 + |\mu(u)|^2 \right)\dvol_\X.
\end{equation}
Here we use the $G$-invariant inner product on $\g$ and the Riemannian metric on $M$ defined as (\ref{eqn:metric-coming-from-omega-and-J}). In the above formula, $|\dA u|^2$ can be calculated as follows. First we note that the complex structure $j$ on $TP/\g^P$ and the volume form $\dvol_\X$ define a metric on $TP/\g^P$: 
\[
 g_\X([v],[w]) = \dvol_\X(v,\tilde{w}) \ \ \text{for}\ \ v,w \in T_pP,
\]
where $\tilde{w}$ is a tangent vector at $p$ satisfying $[\tilde{w}]=j[w]$, and the above definition is well-defined because $\dvol_\X$ is basic. Since $TP/\g^P$ is a vector bundle of rank 2, it is easy to see that $g_\X$ form a metric on $TP/\g^P$. Then define
\begin{equation}
 \label{eqn:norm}
 |\dA u| = \|[v]\|^{-1}\sqrt{|\dA u(v)|^2+|(\dA u \circ j)(v)|^2} 
\end{equation}
for nonzero $[v] \in TP/\g^P$, where $\|[v]\|=\sqrt{g_\X([v],[v])}$. The right hand side is independent of the choice of $[v]$.

\begin{prop}
 For every $(A,u) \in \A(P)\times\Cinf_G(P,M)$, the identity
 \begin{equation}
  \label{eqn:energy-identity}
 E(A,u) = \int_\X \left(\bigl|\brdA u\bigr|^2+\frac{1}{2}\bigl|*F_A+\mu(u)\bigr|^2 \right) 
\dvol_\X + R(\omega,\mu,u)
 \end{equation}
 holds. Here 
 \begin{equation*}
 R(\omega,\mu,u)
  = \int_{\X} \bigl( (\dA u)^*\omega - \pair{\mu(u),F_A} \bigr)
 \end{equation*}
 and $\bigl|\brdA u\bigr|$ is calculated in a similar way to $|\dA u|$.
\end{prop}

\begin{remark}
 We can see as follows that the value $R(\omega,\mu,u)$ is independent of the choice of $A \in \A(P)$. First we consider the Cartan model $\Omega^2_G(P)$ for the $G$-space $P$. Then $u^*(\omega-\mu)$ is a $G$-equivariant 2-form after we identify $\g$ with $\g^\dual$. We can assign the Cartan operator
\[
\Omega_G^2(P) \to \Omega^2([P/G]);\ \eta \to \eta_A
\]
to each connection $A \in A(P)$. In terms of the Cartan operator,
\[
 R(\omega,\mu,u) = \int_\X u^*(\omega-\mu)_A,
\]
and this integration is independent of the choice of the connection $A$ \cite[Proposition 4.2]{cieliebak03:_equiv_euler}. 
\end{remark}

\begin{proof}
We follow the notation in \S \ref{section:orbifold-Riem-surface-revisited}. Let $\chi\in\Cinf(Z_0)$ be a partition of unity. Then
\[
 E(u,A) = \frac{1}{2}\sum_i \int_{U_i} \chi_i \phi_i^* \left( |\dA u|^2 + |F_A|^2 + |\mu(u)|^2 \right)\dvol_\X,
\]
where $\chi_i=\chi|_{U_i}$. Let $s+\i t$ be the standard complex coordinate on $\C$. Then $[d\phi_i(\rd_t)] = j[d\phi_i(\rd_s)]$ and $\phi_i^*A = \Phi_i ds + \Psi_i dt$ for some $\Phi_i, \Psi_i \in \Cinf(U_i,\g)$. Put $u_i = u \circ \phi_i$. Then the pullback of $\dA u$, $F_A$ and $\dvol_\X$ to $U_i$ via $\phi_i$ are given by 
\begin{align*}
 \phi_i^* \dA u &= \left(\frac{\rd u_i}{\rd s} + \Phi_i^M\right)ds + \left(\frac{\rd u_i}{\rd t} + \Psi_i^M\right)dt, \\
 \phi_i^* F_A&= \left(\frac{\rd \Psi_i}{\rd s}-\frac{\rd\Phi_i}{\rd t} + [\Phi_i,\Psi_i]\right)ds \wedge dt, \\
 \phi_i^* \dvol_\X &= \lambda_i^2 ds \wedge dt
\end{align*}
for some positive function $\lambda_i$ on $U_i$. The rest of the proof is similar to Proposition 3.1 in Cieliebak--Gaio--Salamon \cite{cieliebak00:_j_hamil}.
\end{proof}

\begin{cor}
 A pair $(A,u) \in \A(P)\times\Cinf_G(P,M)$ is a solution of the symplectic vortex equation (\ref{eqn:symplectic-vortex-equations}) if and only if $E(A,u)=R(\omega,\mu,u)$.
\end{cor}

Since $\X$ is type-R, the group of automorphisms of $(\pi^P:P\to\X,\sigma)$ in $\PG(\X)$ is anti-isomorphic to the group $\G(P)$ defined as (\ref{eqn:gauge-group}). The group $\G(P)$ acts on the space $\A(P) \times \Cinf_G(P,M)$ by
\[
 \A(P)\times\Cinf_G(P,M) \curvearrowleft \G(P);\ 
 \gamma^*(A,u) = (\gamma^{-1}d\gamma+\gamma^{-1}A\gamma,u\gamma).
\]

\begin{lem}
 The energy functional $E$ is $\G(P)$-invariant.
\end{lem}

This lemma can be shown by direct calculation and implies the following proposition.

\begin{prop}
 The action of $\G(P)$ on $\A(P) \times \Cinf_G(P,M)$ preserves the space of solutions to the symplectic vortex equations (\ref{eqn:symplectic-vortex-equations}).
\end{prop}

\subsection{The case of $G=S^1$ and $M=\C$}
\label{section:circle-actions}

In this subsection, we restrict ourselves to the cases when $G=S^1$ and $M=\C$, and we consider the moduli space of solutions to the symplectic vortex equations by comparing it with the moduli space of K\"ahler vortices \cite{mrowka97:_seiber_witten_seifer}. 

First we fix notation for the circle group $G = S^1 = \{ z \in \C \!\ |\!\ |z|=1 \}$. The Lie algebra $\g=\i\R$ is equipped with an $S^1$-invariant inner product: $\pair{\xi_1,\xi_2} = -\xi_1\xi_2$ for $\xi_1,\xi_2\in\i\R$. Fix a positive integer $a$ we define a left $S^1$-action on $\C$ by
\[
 S^1 \curvearrowright \C;\ t \cdot z = t^{-a} z.
\]
(Therefore the right action is defined as $z \cdot t = t^{-1} \cdot z = t^a z$.)
A moment map is given by 
\[
 \mu: \C \mapsto \i\R;\ z \mapsto \dfrac{\i}{2} \bigl(a|z^2|-\tau\bigr), 
\]
where $\tau$ is a positive real number. Then the Marsden--Weinstein quotient $\MW{\C}S^1$ is called a weighted projective space $\mathbb{P}(a)$ (cf. Sakai \cite[Section 3.3]{sakai12:_delig_mumfor}).

Let $\X$ be a compact orbifold Riemann surface $(\Sigma,j,\mathbf{z},\mathbf{m})$. Choose an object $(\pi:P\to\X,\sigma)$ of $\PG(\X)$. Then $P$ is a 3-dimensional closed manifold equipped with a locally free $S^1$-action, i.e.\! a Seifert fibred space \cite{mrowka97:_seiber_witten_seifer,scott83}. 

The symplectic vortex equations (\ref{eqn:symplectic-vortex-equations}) are given by
\begin{equation}
 \label{eqn:easiest-sve} 
 \brdA u = 0, \quad *F_A + \dfrac{\i}{2} \bigl(a|u|^2-\tau\bigr) = 0,
\end{equation}
where $A \in \A(P) \subset \Omega^1(P,\i\R)$ is a connection on $P$, and $u$ belongs to 
\[
 \Cinf_{S^1}(P,\C) = \{ u \in \Cinf(P,\C) \mid u(pt)=t^a u(p) \text{ for all } p \in P \text{ and } t \in S^1 \}.
\]
The gauge group $\G(P)=\{\gamma \in \Cinf(P,S^1) \mid \gamma(pt)=\gamma(p) \text{ for all } p \in P \text{ and } t \in S^1 \}$ acts on the space of solutions to the equations (\ref{eqn:easiest-sve}) by $\gamma^*(A,u) = (\gamma^{-1}d\gamma+\gamma^{-1}A\gamma,\gamma^a u)$. We denote by $\M(P)$ the moduli space of the solutions to (\ref{eqn:easiest-sve}):
\[
 \M(P) = \{ (A,u) \in \A(P)\times\Cinf_{S^1}(P,\C) \mid (A,u) \text{ satisfies (\ref{eqn:easiest-sve})} \}\big/\G(P).
\]

To consider the moduli space of solutions of the symplectic vortex equations (\ref{eqn:easiest-sve}), we assume the following condition for $\X$.
\begin{assumption}
 \label{assumption:common-multiple}
 Let $\mathbf{m}=(m_1,\dots,m_k)$. Then $a$ is a common multiple of $m_1,\dots,m_k$.
\end{assumption}
The reason for the assumption is as follows. As Remark \ref{remark:adiabatic-limit}, a certain limit of the symplectic vortex equations is the equation for pseudo-holomorphic curve in the Marsden-Weinstein quotient $\MW{\C}S^1$:
\[
 \brdA u = 0, \quad |u|^2=\dfrac{\tau}{a}.
\]
It is natural to assume that the above equations have a solution. Let $(A,u)$ be a solution. For $p \in P$ and a stabiliser $t \in S^1$ at $p$, the map $u:P\to\C$ satisfies $u(p)=t^a u(p)$. By the second equation we have $u(p) \ne 0$, therefore $t^a=1$. This implies the assumption.

\begin{remark}
 If $a=1$, then the orbifold Riemann surface $\X$ must be non-singular. 
\end{remark}

To compare $\M(P)$ with the moduli space of K\"ahler vortices, we describe the moduli space $\M(P)$ by using slices of the $S^1$-space $P$ (cf.\! \S \ref{section:orbifold-Riem-surface-revisited}).

Taking pullbacks via the atlas $\epsilon$ of $P$, we can identify $\Cinf(P,\C)$ with the space 
\[
 \Cinf(Z_1 \times S^1\!\rightrightarrows\!Z_0 \times S^1,\C)
 = \{ \tilde{u} \in \Cinf(Z_0 \times S^1,\C) \mid \src^*\tilde{u}=\tgt^*\tilde{u} \}.
\]
 By this fact, there is a one-to-one correspondence between $\Cinf_{S^1}(P,\C)$ and the space 
\[
 \Cinf(Z_0,\C)_\rho = \{ f \in \Cinf(Z_0,\C) \mid \src^*f = \rho^a \tgt^*f \}. 
\]
The correspondence is explicitly given by $f = u \circ \iota$.

Moreover the pullback $\epsilon^*:\Omega^1(P,\i\R)\to\Omega^1(Z_0 \times S^1,\i\R)$ gives a linear isomorphism from the space $\A(P)$ of connections to the space $\A(Z_1 \times S^1 \!\rightrightarrows\! Z_0 \times S^1)$ defined by 
\begin{align*}
 \{ B \in \Omega^1(Z_0 \times S^1,\i\R) \mid \src^*B=\tgt^*B, \ 
 B(\xi^{Z_0 \times S^1}) = \xi\ (\forall \xi \in \i\R) \ \text{ and }\
 \psi_t^* B = B\ (\forall t \in S^1) \}.
\end{align*}
Let $B_\mathrm{MC}$ be the Maurer--Cartan form for the product bundle $Z_0 \times S^1 \to Z_0$. Define
\[
 \Omega^1(Z_0,\i\R)_\rho = \{
 \theta \in \Omega^1(Z_0,\i\R) \mid 
 \src^*\theta = \rho^{-1}d\rho + \tgt^*\theta
 \}. 
\]
Then the pullback $\pr_1^*:\Omega^1(Z_0,\i\R)\to\Omega^1(Z_0 \times S^1,\i\R)$ gives rise to the linear isomorphism
\[
 \Omega^1(Z_0,\i\R)_\rho\to\A(Z_1 \times S^1 \!\rightrightarrows\! Z_0 \times S^1);\
 \theta \mapsto B_\mathrm{MC} + \pr_1^*\theta.
\]
Therefore we obtain a one-to-one correspondence between $\A(P)$ and $\Omega^1(Z_0,\i\R)_\rho$. The explicit correspondence is given by $\epsilon^*A=B_\mathrm{MC}+\pr_1^*\theta$ (or $\iota^*A=\theta$) for $A \in \A(P)$ and $\theta \in \Omega^1(Z_0,\i\R)_\rho$.

The gauge group $\G(P)$ can be identifies with 
\[
 \Cinf(\gpoid{Z},S^1)= \{ g \in \Cinf(Z_0,S^1) \mid \src^*g = \tgt^*g  \}
\]
via the identification $\gamma \circ \iota = g$ for $\gamma \in \G(P)$ and $g \in \Cinf(\gpoid{Z},S^1)$ and the action of the gauge group $\G(P)$ on $\A(P) \times \Cinf_{S^1}(P,\C)$ can be identified with the following action:
\[
 \Omega^1(Z_0,\i\R)_\rho \times \Cinf(Z_0,\C)_\rho \curvearrowleft \Cinf(\gpoid{Z},S^1);\
 (\theta,f) \cdot g = (g^{-1}dg+\theta,g^af).
\]

Under the above identification, the equations (\ref{eqn:easiest-sve}) can be identified the following PDEs by taking pullbacks via $\iota:Z_0 \to P$:
\begin{equation}
 \label{eqn:sve-on-Z0} 
 \overline{\partial}_{-a\theta} f = 0, \quad *F_\theta + \dfrac{\i}{2} \bigl(a|f|^2-\tau\bigr) = 0
\end{equation}
for $(\theta,f) \in \Omega^1(Z_0,\i\R)_\rho \times \Cinf(Z_0,\C)_\rho$.
Here, for the $*$-operator on $Z_0$, the volume form on $Z_0$ is the pullback $\iota^*\dvol_\X$. Therefore $\M(P)$ is identified with 
\[
 \M(\gpoid{Z}) = \{ (\theta,f) \in \Omega^1(Z_0,\i\R)_\rho \times \Cinf(Z_0,\C)_\rho
 \mid (\theta,f) \text{ satisfies (\ref{eqn:sve-on-Z0})} \}\big/\Cinf(\gpoid{Z},S^1).
\]

Next we describe the moduli space of K\"ahler vortices in terms of $\gpoid{Z}$. The $S^1$-action on $P \times \C$ by $(p,x) \cdot t = (pt,t^ax)$ is proper and locally free, and therefore the action induces an orbifold structure on the quotient space $P \times_{S^1} \C$. Moreover the quotient map $P\times_{S^1}\C\to P/S^1 = \Sigma$ gives rise to a Hermitian orbifold line bundle. The Hermitian orbifold line bundle can be described as the following proper \'etale groupoid
\[
 Z_1\times\C \rightrightarrows Z_0\times\C;\ \
 \begin{cases}
  \src(\alpha,x) = (\src\alpha,x), \\
  \tgt(\alpha,x) = (\tgt\alpha,\rho(\alpha)^ax).
 \end{cases}
\]

The space of orbisections is naturally identified with $\Cinf(Z_0,\C)_\rho$, and the space of the compatible orbifold connections is given by
\[
\Omega^1(Z_0,\i\R)_a =
 \{ B \in \Omega^1(Z_0,\i\R) \mid \src^*B = -a \rho^{-1}d\rho + \tgt^*B \}.
\]
The gauge group is nothing but $\Cinf(\gpoid{Z},S^1)$ and the (right) action of the gauge group on $\Omega^1(Z_0,\i\R)_a \times \Cinf(Z_0,\C)_\rho$ is defined by
\[
 \Omega^1(Z_0,\i\R)_a \times \Cinf(Z_0,\C)_\rho \curvearrowleft \Cinf(\gpoid{Z},S^1);\ (B,f) \cdot g = (g^{-1}dg+B,g^{-1}f).
\]
The equations of K\"aher vortices over the orbifold Riemann surface $(\Sigma,j,\mathbf{z},\mathbf{m})$ are the PDEs
\begin{equation}
 \label{eqn:Kahler-vortices} 
 \overline{\partial}_{B} f = 0, \quad 
 *F_{B} + \dfrac{\i}{2} \bigl(|f|^2-a\tau\bigr) = 0
\end{equation}
for $(B,f) \in \Omega^1(Z_0,\i\R)_a \times \Cinf(Z_0,\C)_\rho$.

\begin{remark}
 The above equations are obtained after we replace $F_{K_\Sigma}$ with $-a\i\tau\dvol_\X$ in the equations of K\"ahler vortices \cite[Section 5.5]{mrowka97:_seiber_witten_seifer}. This replacement does not produce any problems on the following discussion (Proposition \ref{prop:Theorem-5-of-MOY}) in this paper.
\end{remark}

Therefore the moduli space of K\"ahler vortices is defined by 
\[
 \M^\mathrm{K} = \{ (B,f) \in \Omega^1(Z_0,\i\R)_a \times \Cinf(Z_0,\C)_\rho
 \mid (B,f) \text{ satisfies (\ref{eqn:Kahler-vortices})} \}\big/\Cinf(\gpoid{Z},S^1).
\]

Mrowka, Ozsv\'ath and Yu show that the moduli space of K\"aher vortices has a structure of complex manifold and the complex dimension is equal to the integer called the background degree. To calculate it, we recall briefly Seifert invariants. (The details of Seifert invariants can be found in Mrowka--Ozsv\'ath--Yu \cite[Section 2]{mrowka97:_seiber_witten_seifer} and Furuta--Steer \cite{furuta92:_seifer_yang_mills_rieman}.) 

Orbifold line bundles over an orbifold Riemann surface can be classified by pairs of integers called Seifert invariants:
\begin{equation}
 \label{eqn:Seifert-invariant}
 \mathbf{b} = (b,\beta_1,\dots,\beta_k).
\end{equation}
Here $k$ is the number of singular points of the base orbifold, $\beta_i$ is an integer with $0 \leq \beta_i < m_i$, and $b$ is the integer which is called the background degree of the orbifold line bundle. The integer $\beta_i$ is defined by local structure of the orbifold line bundle around the $i$-th singular point. The background degree $b$ is the degree of the de-singularisation of the orbifold line bundle. For an orbifold line bundle $E \to \Sigma$ with the Seifert invariant (\ref{eqn:Seifert-invariant}), the degree of the line bundle is calculated by the formula
\begin{equation}
 \label{eqn:degree-by-Seifert-invariant}
  \deg(E) = b + \dfrac{\beta_1}{m_1} + \cdots + \dfrac{\beta_k}{m_k}.
\end{equation}

Define $d$ as the real number $\displaystyle\frac{\i}{2\pi}\int_{[P/S^1]} F_A$ which is independent of the choice of $A \in \A(P)$. 
\begin{lem}
 The Seifert invariant of the orbifold line bundle $P\times_{S^1}\C\to\Sigma$ is
$\mathbf{b}= \left(ad,0,\dots,0\right)$.
\end{lem}
\begin{proof}
 Note that the triple $(U_\alpha\times\C,\phi_\alpha\times\id_\C,\Gamma_\alpha)$ is a slice for the $S^1$-action on $P\times\C$. Here the right action of $\Gamma_\alpha$ on $U_\alpha\times\C$ is given by
\[
 (z,x) \cdot t = (zt,t^ax) = (zt,x).
\]
The last identity follows Assumption \ref{assumption:common-multiple}, and therefore $\beta_1=\dots=\beta_k=0$ in the description of Seifert invariant (\ref{eqn:Seifert-invariant}). 

Because of the identity (\ref{eqn:degree-by-Seifert-invariant}), the background degree is equal to the degree of the orbifold line bundle $P\times_{S^1}\C\to\Sigma$. Direct calculation shows that the degree is equal to $ad$.
\end{proof}

\begin{remark}
 The above lemma implies that $ad$ is an integer under Assumption \ref{assumption:common-multiple}.
\end{remark}

Applying Theorem 5 of Mrowka--Ozsv{\'a}th--Yu \cite{mrowka97:_seiber_witten_seifer}, we obtain the following statement.

\begin{prop}
 \label{prop:Theorem-5-of-MOY}
 The moduli space $\M^\mathrm{K}$ of K\"ahler vortices is empty if 
 \[
  d > \frac{\tau\mathrm{Vol}(\Sigma)}{4\pi},
 \]
 and it is naturally diffeomorphic to the $ad$-fold symmetric product of $\Sigma$, $\mathrm{Sym}^{ad}(\Sigma)$ if 
 \begin{equation}
  \label{eqn:nonempty-moduli}
  d < \frac{\tau\mathrm{Vol}(\Sigma)}{4\pi}.
 \end{equation}
\end{prop}

Let $\tilde{\M}(\gpoid{Z})$ and $\tilde{\M}^\mathrm{K}$ be the spaces of solutions to (\ref{eqn:sve-on-Z0}) and (\ref{eqn:Kahler-vortices}), respectively. By definition $\M(\gpoid{Z})=\tilde\M(\gpoid{Z})\big/\Cinf(\gpoid{Z},S^1)$ and $\M^\mathrm{K}=\tilde\M^\mathrm{K}\big/\Cinf(\gpoid{Z},S^1)$ hold. Now we compare $\M(\gpoid{Z})$ with $\M^\mathrm{K}$. The following linear isomorphism
\[
\Psi: \Omega^1(Z_0,\i\R)_\rho \times \Cinf(Z_0,\C)_\rho \to \Omega^1(Z_0,\i\R)_a \times \Cinf(Z_0,\C)_\rho;\ (\theta,f) \mapsto (-a\theta,af)
\]
gives a one-to-one correspondence between $\tilde\M(\gpoid{Z})$ and $\tilde\M^\mathrm{K}$. Moreover the map $\Psi$ satisfies
\begin{equation}
 \Psi\bigl((\theta,f) \cdot g \bigr) = \Psi(\theta,f) \cdot g^{-a}
\end{equation}
for $\theta \in \Omega^1(Z_0,\i\R)_\rho$, $f \in \Cinf(Z_0,\C)_\rho$ and $g \in \Cinf(\gpoid{Z},S^1)$. Therefore $\Psi$ induces a well-defined surjective map $\overline\Psi: \M(\gpoid{Z}) \to \M^\mathrm{K}$.

In general, the map $\overline\Psi$ is not injective. Define a covering map $\sigma:S^1 \to S^1$ by $\sigma(t)=t^a$ and a map $\tilde\sigma:\Cinf(\gpoid{Z},S^1)\to\Cinf(\gpoid{Z},S^1)$ by $\tilde\sigma(g)=\sigma \circ g$. Then $\Psi$ induces a bijection between $\M(\gpoid{Z})$ and $\tilde\M^\mathrm{K}\big/\image\tilde\sigma$, and each fibre of the natural projection $\tilde\M^\mathrm{K}\big/\image\tilde\sigma \to \tilde\M^\mathrm{K}\big/\Cinf(\gpoid{Z},S^1) = \M^\mathrm{K}$ can be identified with $\coker\tilde\sigma$.
\[
 \xymatrix@C=60pt@R=20pt{
 \M(\gpoid{Z}) \ar@{<->}[r]^{1:1} \ar[dr]_{\overline\Psi}& 
 \tilde\M^\mathrm{K}\big/\image\tilde\sigma
 \ar[d]
 \\ & 
 \M^\mathrm{K}.
 }
\]

\begin{remark}
 If the inequality (\ref{eqn:nonempty-moduli}) holds, then the action of $\Cinf(\gpoid{Z},S^1)$ on $\tilde\M^\mathrm{K}$ is free.
\end{remark}

In general, the cokernel of $\tilde\sigma$ is complicated, but it is trivial if the underlying space $\Sigma$ of $\X$ is simply-connected.

\begin{lem}
 If the underlying space $\Sigma$ is simply-connected (i.e.\! of genus zero), then for any $g \in \Cinf(\gpoid{Z},S^1)$ there is $g_0 \in \Cinf(\gpoid{Z},S^1)$ such that $g = g_0^{a}$.
\end{lem}
\begin{proof}
 Let $\pi:Z_0 \to Z_0/Z_1 = \Sigma$ be the quotient map. Any $g \in \Cinf(\gpoid{Z},S^1)$ induces a continuous map $h:\Sigma \to S^1$ such that $g = h \circ \pi$. Since $\Sigma$ is simply-connected, $h$ has a lift: a continuous map $\tilde{h}:\Sigma \to S^1$ satisfying $\sigma \circ \tilde{h}=h$.

 Since $\sigma\circ(\tilde{h}\circ\pi)=h\circ\pi=g$, the continuous map $\tilde{h}\circ\pi:Z_0 \to S^1$ is a lift of $g:Z_0 \to S^1$, hence the composition $\tilde{h}\circ\pi$ is smooth. Since $\tilde{h}\circ\pi\circ\src=\tilde{h}\circ\pi\circ\tgt$, $\tilde{h}\circ\pi \in \Cinf(\gpoid{Z},S^1)$. The map $g_0 = \tilde{h}\circ\pi$ is what we want. 
\end{proof}

Identifying $\M(P)$ with $\M(\gpoid{Z})$, we obtain the following statement.

\begin{thm}
 If $\Sigma$ is simply-connected, then the moduli space $\M(P)$ is empty if 
 \[
  d > \frac{\tau\mathrm{Vol}(\Sigma)}{4\pi},
 \]
 and it has a smooth structure which is diffeomorphic to the complex projective space ${\C}P^{ad}$ if
 \[
  d < \frac{\tau\mathrm{Vol}(\Sigma)}{4\pi}.
 \]
\end{thm}

\paragraph{Acknowledgments}
The author would like to thank the Max-Planck-Institut f\"ur Mathematik in Bonn for providing financial support and excellent environment during my stay. He would also like to thank Andreas Ott and Nuno Miguel Rom\~ao for valuable comments and discussions.


\medskip
{\em
\noindent
Max-Planck-Institut f\"ur Mathematik \newline
Vivatsgasse 7 \newline
53111 Bonn \newline
GERMANY

\noindent
E-mail: sakai@blueskyproject.net
}
\end{document}